\documentclass[11pt]{amsart}

\usepackage{stmaryrd}
\usepackage{amscd,amsxtra,amssymb,mathrsfs, bbm}

\usepackage{multirow}
\usepackage[pdftex]{hyperref}
\usepackage{longtable}
\usepackage{tikz}
\usepackage{float}
\usepackage{color}
\usetikzlibrary{calc, positioning, shapes, fit, matrix, decorations}
\usetikzlibrary{decorations.shapes}

\usepackage[cmtip]{xypic}
\usepackage[all]{xy}
\xyoption{arc}

\newtheorem{theorem}{Theorem}[section]
\newtheorem{corollary}[theorem]{Corollary}
\newtheorem{lemma}[theorem]{Lemma}
\newtheorem{proposition}[theorem]{Proposition}

\theoremstyle{definition}
\newtheorem{definition}[theorem]{Definition}
\newtheorem{remark}[theorem]{Remark}
\newtheorem{example}[theorem]{Example}

\numberwithin{equation}{section}

\DeclareMathAlphabet{\mathpzc}{OT1}{pzc}{m}{it}

\DeclareMathOperator{\Perf}{\mathsf{Perf}}

\renewcommand{\ker}{\mathsf{ker}}
\newcommand{\im}{\mathsf{im}}

\DeclareMathOperator{\tor}{\mathsf{tor}}

\DeclareMathOperator{\Tri}{\mathsf{Tri}}

\DeclareMathOperator{\can}{\mathsf{can}}

\DeclareMathOperator{\CM}{\mathsf{CM}}
\DeclareMathOperator{\MF}{\mathsf{MF}}

\DeclareMathOperator{\Rep}{\mathsf{Rep}}

\DeclareMathOperator{\Ob}{\mathsf{Ob}}

\DeclareMathOperator{\depth}{\mathsf{depth}}

\DeclareMathOperator{\Hom}{\mathsf{Hom}}

\DeclareMathOperator{\Ext}{\mathsf{Ext}}

\DeclareMathOperator{\GL}{\mathsf{GL}}
\DeclareMathOperator{\Aut}{\mathsf{Aut}}
\DeclareMathOperator{\Ann}{\mathsf{ann}}
\DeclareMathOperator{\End}{\mathsf{End}}
\DeclareMathOperator{\Mat}{\mathsf{Mat}}

\input xy
\xyoption{all}

\newcommand{\bx}{\mathbf{x}}
\newcommand{\by}{\mathbf{y}}
\newcommand{\bz}{\mathbf{z}}

\newcommand{\kk}{\mathbbm{k}}

\newcommand{\llbrace}{(\!(}
\newcommand{\rrbrace}{)\!)}

\newcommand{\FF}{\mathbb{F}}
\newcommand{\GG}{\mathbb{G}}

\newcommand{\II}{\mathbb{I}}
\newcommand{\JJ}{\mathbb{J}}
\renewcommand{\mod}{\mathsf{mod}}

\newcommand{\lar}{\longrightarrow}

\newcommand{\CA}{\mathsf A} 
\newcommand{\RD}{\mathsf D}
\newcommand{\RP}{\mathsf P}
\newcommand{\RQ}{\mathsf Q}
\newcommand{\RT}{\mathsf T}
\newcommand{\rA}{R}
\newcommand{\rR}{S}

\DeclareMathOperator{\Tau}{T}

\newcommand{\RA}{\mathsf R}
\newcommand{\RR}{\mathsf S}

\newcommand{\idm}{\mathfrak{m}}
\newcommand{\idn}{\mathfrak{n}}
\newcommand{\idp}{\mathfrak{p}}

\newcommand{\mM}{N}

\newcommand{\DD}{\mathbb{D}}
\newcommand{\PP}{\mathbb{P}}

\def\dE{\mathfrak E}		\def\dF{\mathfrak F}		
\def\dX{\mathfrak B}

\def\rep{\mathop\mathsf{Rep}}

\def\Mat{\mathop\mathrm{Mat}}
\def\End{\mathop\mathrm{End}\nolimits}

\def\8{\infty}			
	\def\+{\oplus}		
\def\*{\otimes}

	\def\NN{\mathbb N}

\def\dX{\mathfrak B}

\def\DMO{\DeclareMathOperator}
\DMO{\ob}{Ob}            \DMO{\mor}{Mor}
\DMO{\Ker}{Ker}
\DMO{\id}{Id}

\newcommand{\hdotline}[3]{\draw[dotted] (#1-#2-1.south west) -- (#1-#2-#3.south east);}
\newcommand{\vdotline}[3]{\draw[dotted] (#1-1-#3.north east) -- (#1-#2-#3.south east);}
\newcommand{\hdline}[3]{\draw[dashed] (#1-#2-1.south west) -- (#1-#2-#3.south east);}

\newcommand{\hsline}[3]{\draw[very thick] (#1-#2-1.south west) -- (#1-#2-#3.south east);}

\newcommand{\sfrm}[3]{
\node[draw,solid, thick, fit=(#1-1-1)(#1-#2-#3), inner sep=0pt]{};}

\tikzset{
    tbl5 nodes/.style={
        rectangle,
        execute at begin node=$,
       execute at end node=$,
       fill=blue!5,
        align=center,
        text depth=0.5ex,
        text height=2ex,
        inner xsep=0pt,
        outer sep=0pt,
           },
    tbl5/.style={
        matrix of nodes,
        row sep=-\pgflinewidth,
        column sep=-\pgflinewidth,
        nodes={
            tbl5 nodes
        },
        execute at empty cell={\node[draw=none]{};}
    }
  }

\tikzset{
    table nodes/.style={
        rectangle,
        fill=blue!5,
        draw,
        thin, dashed,
        align=center,
        minimum height=20pt,
        text depth=0.5ex,
        text height=2ex,
        inner xsep=0pt,
        outer sep=0pt
    },
    table/.style={
        matrix of nodes,
        row sep=-\pgflinewidth,
        column sep=-\pgflinewidth,
        nodes={
            table nodes
        },
        execute at empty cell={\node[draw=none]{};}
    }
  }

\tikzset{
    table2 nodes/.style={
        rectangle,
        draw,
        thin, dashed,
        align=center,
        text depth=0.5ex,
        text height=2ex,
        inner xsep=0pt,
        outer sep=0pt
    },
    table2/.style={
        matrix of nodes,
        row sep=-\pgflinewidth,
        column sep=-\pgflinewidth,
        nodes={
            table2 nodes
        },
        execute at empty cell={\node[draw=none]{};}
    }
  }

\title[Cohen--Macaulay modules over non--reduced curve singularities]{Cohen--Macaulay modules over some non--reduced curve singularities}

\author{Igor Burban}
\address{
Universit\"at zu K\"oln,
Mathematisches Institut,
Weyertal 86-90,
D-50931 K\"oln,
Germany
}
\email{burban@math.uni-koeln.de}

\author{Wassilij Gnedin}
\address{
Universit\"at zu K\"oln,
Mathematisches Institut,
Weyertal 86-90,
D-50931 K\"oln,
Germany
}
\email{wgnedin@math.uni-koeln.de}

\begin{document}
\maketitle

\begin{abstract} In this article, we study Cohen--Macaulay modules over non--reduced curve singularities. We prove that the rings
$\kk \llbracket x,y,z \rrbracket / (xy, y^q - z^2)$
have  tame  Cohen--Macaulay  representation type.
For  the  singularity $\kk \llbracket x,y,z \rrbracket / (xy, z^2)$ we give an explicit description
of all indecomposable Cohen--Macaulay modules and apply the obtained classification
to construct  explicit families of indecomposable matrix
factorizations of $x^2 y^2 \in \kk\llbracket x, y\rrbracket$.
\end{abstract}
\maketitle

\section*{Introduction}
Cohen--Macaulay modules over Cohen--Macaulay rings have been   intensively studied in recent years.
They appear in the literature in various incarnations  like matrix factorizations, objects of the triangulated
category of singularities
or lattices over orders.

Our interest to Cohen--Macaulay modules is representation theoretic. In the case of a \emph{reduced} curve singularity, the behavior of the representation type of the category of Cohen--Macaulay modules $\CM(A)$ is completely understood. Assume, for simplicity, that $A$ is an algebra over an algebraically closed field $\kk$ of characteristic zero.

\begin{itemize}
\item According to Drozd and Roiter \cite{DroRoi}, Jacobinski \cite{Jacob} and Greuel and Kn\"orrer \cite{GreuelKnoerrer},
$\CM(A)$ is representation finite if and only if $A$ dominates a simple curve singularity.
See also the expositions in the monographs \cite{LeuschkeWiegand} and \cite{Yoshino}.
\item Drozd and Greuel have proven in \cite{DG}
that $\CM(A)$ is tame if and only
if $A$ dominates a singularity of type
\begin{align*}
&&  \RT_{pq}(\lambda) = \kk \llbracket x,y \rrbracket / (x^{p-2} - y^2)(x^2 - \lambda y^{q-2})&  & \begin{cases} \tfrac{1}{p} + \tfrac{1}{q} = \tfrac{1}{2}, & \lambda \in
\kk\backslash\{0,1\}, \\
\tfrac{1}{p} + \tfrac{1}{q} < \tfrac{1}{2}, & \lambda = 1.
\end{cases}
 \end{align*}
In particular, they have shown that the singularities
\begin{align*}
\RP_{pq} = \kk\llbracket x,y,z\rrbracket/(xy, x^p + y^q - z^2), \qquad \qquad \mbox{where }   p, q  \in \mathbb{N}_{\ge 2}, \quad
\end{align*}
are Cohen--Macaulay tame.
\item A reduced curve singularity which neither dominates a simple nor a $\RT_{pq}(\lambda)$ singularity has wild Cohen--Macaulay representation type \cite{DG-92}.
\item There are also other approaches  to establish  tameness  of $\CM\bigl(\RT_{pq}(\lambda)\bigr)$: one
 using  the generalized geometric McKay Correspondence
\cite{Kahn, DGK} and another via  cluster--tilting theory \cite{BIKR}.
\end{itemize}

\noindent
The following results about the representation type of a non--reduced curve singularity are known so far.
\begin{itemize}
\item By a theorem of Auslander \cite{Auslander}, a non--reduced curve singularity always  has infinite
 Cohen--Macaulay representation type.
\item Buchweitz, Greuel and Schreyer have shown in \cite{BGS} that the  singularities $\CA_\infty = \kk\llbracket x, y\rrbracket/(y^2)$ and $\RD_\infty = \kk\llbracket x,y\rrbracket/(x y^2)$ have discrete Cohen--Macaulay representation type.
\item Leuschke and Wiegand have proven in \cite{LW-05} that
$\CA_\infty$, $\RD_\infty$ and  $\kk\llbracket x,y,z\rrbracket/(xy, yz, z^2)$ are the only curve singularities of bounded but infinite Cohen--Macaulay type.
\item Burban and Drozd have proven in \cite{BDNonIsol} that the hypersurface singularities $\RT_{\infty q} = \kk\llbracket x, y\rrbracket/(x^2 y^2 - y^q)$, where $q \in \mathbb{N}_{\ge 3}$,
(respectively
 $\RT_{\infty \infty} = \kk\llbracket x, y\rrbracket/(x^2 y^2)$) are Cohen--Macaulay tame (under the additional assumption that $\mathsf{char}(\kk) = 0$,  respectively $\mathsf{char}(\kk) \ne 2$).
 However, an explicit description of the corresponding indecomposable matrix factorizations is still not known.
\end{itemize}

\noindent
In this article, we obtain  the following results.

\medskip
\noindent
 1.~First, we prove (see Theorem \ref{T:main1}) that the curve singularities
$$
\RP_{\infty q} = \kk\llbracket x, y, z\rrbracket/(xy, y^q - z^2) \quad \mbox{ and} \quad
\RP_{\infty \infty} = \kk\llbracket x,y,z\rrbracket/(xy, z^2)
$$ are Cohen--Macaulay tame for any algebraically closed field $\kk$ of any characteristic (in the case $\mathsf{char}(\kk) = 2$ the definition
of $\RP_{\infty q}$ has to be modified, see Remark \ref{R:RonPQchar2}). The method of the proof extends the approach of Drozd
and Greuel \cite{DG} to the case of non--reduced curve singularities and is based on Bondarenko's work on representations of
\emph{bunches of semi--chains} \cite{bo1}. Our approach can be summarized by the following diagram of categories and functors:
$$
\xymatrix@M+1pt{
\CM(\RP)
& \CM(\RA) \ar@/^/[r]^-{\FF} \ar@{}[r]|-{\sim} \ar@{_{(}->}[l]_-{\II}  & \Tri(\RA) \ar[r]^-{\PP} \ar@/^/[l]^-{\GG} & \Rep(\dX). \\
}
$$
We start with a singularity $\RP = \RP_{\infty q}$ or $\RP_{\infty\infty}$ and replace it by its \emph{minimal overring} $\RA$. The forgetful functor
 $\II$ embeds $\CM(\RA)$ into
$\CM(\RP)$ as a full subcategory. By a result of Bass \cite{Bass}, the  ``difference'' between $\CM(\RA)$ and $\CM(\RP)$
 is very small. The \emph{category of triples}
$\Tri(\RA)$ plays a key role in our approach.  According to \cite{BDNonIsol}, the functors
 $\FF$ and $\GG$ are quasi--inverse equivalences of categories. Finally, $\Rep(\dX)$ is a certain \emph{bimodule category} in the sense of \cite{DrozdLOMI}.
The functor $\PP$ preserves isomorphy classes and indecomposability of objects.
  We prove that $\Rep(\dX)$ is the category of representations of a certain \emph{bunch of semi--chains}.
According to a theorem of Bondarenko \cite{bo1},  $\Rep(\dX)$ is representation tame. This implies tameness of $\CM(\RP)$.

\medskip
\noindent
2.~Next, we show how to  pass
from  canonical forms  describing indecomposable objects of $\Rep(\dX)$
 to a concrete description of the corresponding indecomposable Cohen--Macaulay $\RP$--modules.
 We illustrate this technique giving  an explicit description of the indecomposable Cohen--Macaulay modules over $\RP_{\infty\infty}$. They  are described  in terms of quite transparent combinatorial data: \emph{bands} and \emph{strings}, see Theorem \ref{C:main}. The obtained classification  turns out to be perfectly adapted to separate those Cohen--Macaulay modules which are \emph{locally free on the punctured spectrum} from those which are not, see Remark \ref{R:locfreePSpec}.

\medskip
\noindent
3.~At last, we construct   explicit families of indecomposable matrix factorizations of $x^2 y^2 \in \kk\llbracket x, y\rrbracket$.
In this context, there is the following diagram of categories and functors:
$$
\xymatrix@M+1pt{
\CM(\RA)
\ar@{^{(}->}[r]^-{\JJ}
& \CM(\RT) \ar[r]   & \underline{\MF}(x^2 y^2).
}
$$
Here, $\RA$ is the minimal overring of $\RP_{\infty \infty}$, the functor $\JJ$ is a fully faithful embedding, $\RT = \kk\llbracket x, y\rrbracket/(x^2 y^2)$ and $\underline{\MF}(x^2 y^2)$ is the homotopy
category of  matrix factorizations of $x^2 y^2$ (which is equivalent to the stable category $\underline{\CM}(\RT)$ by a result
of  Eisenbud \cite{Eisenbud}).
Results of this article provide a partial classification of the indecomposable objects of $\underline{\MF}(x^2 y^2)$ as well as an equivalent category $\underline{\MF}(x^2 y^2 + uv)$.

\medskip
\noindent
\emph{Acknowledgement}. This research was   supported by the DFG
project   Bu--1866/2--1.  We are also thankful  to Lesya Bodnarchuk for supplying us with TikZ pictures illustrating the technique of matrix problems.

\section{Survey  on Cohen--Macaulay modules over curve singularities}

\noindent
In this section we collect definitions and some known facts on Cohen--Macaulay modules over curve singularities. The proofs of the mentioned statements can be found in the monographs \cite{BrunsHerzog,LeuschkeWiegand,Yoshino}, see also the survey article \cite{SurvOnCM}.

\subsection{Definitions and basic properties} Let $(A, \idm)$ be a local Noetherian ring of Krull dimension one (a curve singularity), $\kk = A/\idm$ its residue field and $Q = Q(A)$ its total ring of fractions.

\begin{definition}
A curve singularity $A$  is
\begin{itemize}
\item \emph{Cohen--Macaulay} if and only if $\Hom_{A}(\kk, A) = 0$ (equivalently, $A$ contains a regular element).
\item \emph{Gorenstein} if and only if it is Cohen--Macaulay and $\Ext^1_A(\kk, A) \cong \kk$ (equivalently,
$\mathsf{inj.dim}_A(A) = 1$).
\end{itemize}
\end{definition}

\noindent
Note that a \emph{reduced} curve singularity is automatically Cohen--Macaulay. However, in this article we mainly focus on  non--reduced ones.

\begin{lemma}\label{L:RingofFractions}
Let $A$ be a Cohen--Macaulay curve singularity. Then $Q$ is an Artinian ring. Moreover,  if  $\bigl\{\idp_1, \dots, \idp_t\bigr\}$ is the set of minimal
prime ideals of $A$ then there exists a ring  isomorphism $\gamma\!: Q \lar A_{\idp_1} \times \dots \times A_{\idp_t}$ making the following diagram
$$
\xymatrix{
 & Q \ar[dd]^-{\gamma} \\
A \ar[ru]^-{\can_1} \ar[rd]_-{\can_2} & \\
 &  A_{\idp_1} \times \dots \times A_{\idp_t}
}
$$
commutative,
where $\can_1$ and $\can_2$ are canonical morphisms.
\end{lemma}
\begin{proof}
Since $A$ is Cohen--Macaulay, the associator of $A$ coincides with $\{ \idp_1, \ldots, \idp_n\}$.
By \cite[Chapter IV, Proposition 2.5.10]{Bourbaki} $Q$ is Artinian and its maximal ideals
are $\idp_1 Q, \ldots, \idp_n Q$. Hence, $Q \cong  Q_{\idp_1 Q} \times \dots \times Q_{\idp_n Q}$. Since
  $Q_{\idp_i Q} = A_{\idp_i}$ for  $1 \leq i \leq n$, the result follows.
\end{proof}

\begin{definition} For an $A$--module $M$ we set
$$
\Gamma_{\idm}(M) := \bigl\{ \ x \in M \ \big| \ \idm^t x = 0 \, \mbox{ for some } \, t \in \NN \ \bigr\}.
$$
\end{definition}

\noindent
The following result can be easily deduced  from Lemma \ref{L:RingofFractions}.

\begin{lemma}\label{L:LemmaCM}
Let $A$ be a Cohen-Macaulay curve singularity.
For a Noetherian $A$--module $M$ we have:
$$
\Gamma_{\idm}(M) = \ker\bigl(M \lar Q \otimes_A M\bigr) =: \tor(M).
$$
Moreover, the following  statements  are equivalent:
\begin{itemize}
\item $\Hom_A(\kk, M) = 0$.
\item $M$ is \emph{torsion free}, i.e. $\tor(M) = 0$.
\end{itemize}
\end{lemma}

\begin{definition}
A Noetherian module $M$ satisfying  the  conditions of Lemma \ref{L:LemmaCM} is called \emph{maximal Cohen--Macaulay}. In what follows we just say that $M$ is Cohen--Macaulay.
In this case, the  $Q$--module $Q(M)$ is called the \emph{rational envelope} of $M$.  More generally, a Noetherian module $N$ over a Noetherian ring $S$
 (say, of Krull dimension one)  is (maximal) Cohen--Macaulay
if for any maximal ideal $\idn$ in $S$ the localization $N_\idn$ is Cohen--Macaulay. In what follows, $
\CM(S)$ denotes the category of Cohen--Macaulay
 $S$--modules.
\end{definition}

\begin{lemma}\label{L:CMsubmFree}
Assume that a Cohen--Macaulay curve singularity $A$ is Gorenstein in codimension zero (i.e.~$Q$ is self--injective).
Then for a Noetherian $A$--module $M$, the following conditions are equivalent:
\begin{itemize}
\item $M$ is Cohen--Macaulay.
\item $M$ embeds into a finitely generated free $A$--module.
\end{itemize}
\end{lemma}
\noindent
A proof of this Lemma can be found in \cite[Appendix A, Corollary 15]{LeuschkeWiegand}.
\begin{remark}
The statement of Lemma \ref{L:CMsubmFree} is not true for an arbitrary  Cohen--Macaulay curve singularity.
For example, let $A = \kk\llbracket x, y, z\rrbracket/(x^2, xy, y^2)$ and $K$ be a canonical $A$--module.
Then $K$ does not embed into a free $A$--module.
\end{remark}

\begin{definition}
A ring $R$ is an \emph{overring} of $A$ if $A \subseteq R \subset Q$ and the ring extension $A \subseteq R$ is finite.
We also say that $R$ \emph{birationally dominates} $A$.
\end{definition}

\begin{proposition}\label{P:overrings}
Let $A$ be a Cohen--Macaulay curve singularity and $R$ an overring of $A$. Then the following results are true.
\begin{itemize}
\item $R$ is Cohen--Macaulay.
\item We have an adjoint pair $\bigl(R \boxtimes_A \,-\,, \II(\,-\,)\bigr)$, where $\II\!: \CM(R) \lar \CM(A)$ is the restriction (or forgetful) functor and
$R \boxtimes_A \,-\, \!: \CM(A) \lar \CM(R)$ sends a Cohen--Macaulay module $M$ to $R \otimes_A M/\tor(R \otimes_A M)$.
\item $\II$ is fully faithful.
\item If $M = \bigl\langle w_1, \dots, w_t\bigr\rangle_A \subset Q^n$, then $R \boxtimes_A M \cong R \cdot M := \bigl\langle w_1, \dots, w_t\bigr\rangle_R \subset Q^n$.
\end{itemize}
\end{proposition}
\begin{proof}
The first statement follows from the fact that $\depth_A(R) = 1 = \depth_R(R)$.
The second result follows from the functorial isomorphisms
$$\Hom_R\bigl(R \boxtimes_A M ,N\bigr) \cong \Hom_R\bigl(R \otimes_A M, N\bigr) \cong \Hom_A\bigl(M,\II(N)\bigr).
$$
For a proof of the third statement, see for example \cite[Lemma 4.14]{LeuschkeWiegand}. The fourth result follows from the fact that  the kernel of the natural morphism
$R \otimes_A M \longrightarrow R\cdot M$ is  $\tor(R \otimes_A M)$.
\end{proof}

\begin{corollary}
Let $A$ be a Cohen--Macaulay curve singularity and $R$ be an overring of $A$. Then the following statements are true.
\begin{itemize}
\item Let $N_1$ and $N_2$ be Cohen--Macaulay $R$--modules. Then $N_1 \cong N_2$ if and only if $\II(N_1) \cong \II(N_2)$ in
$\CM(A)$.
\item A Cohen--Macaulay $R$--module $N$ is indecomposable if and only if $N$  is indecomposable viewed
as an $A$--module.
\end{itemize}
\end{corollary}

\noindent
The following result is due to  Bass \cite[Proposition 7.2]{Bass}, see also
\cite[Lemma 4.9]{LeuschkeWiegand}.

\begin{theorem}\label{T:RejectionLemma}
Let $(A, \idm)$ be a Gorenstein curve singularity and let
$R = \End_A(\idm)$. Then the following results are true.
\begin{itemize}
\item $R \cong  \bigl\{r \in Q \ \big|\, r \, \idm \subseteq \idm \bigr\}$. In particular, $R$ is an overring of $A$.
\item If $A$ is not regular, we have an exact sequence of $A$--modules
$$
0 \lar A \stackrel{\imath}\lar R \lar \kk \lar 0,
$$
where $\imath$ is the canonical inclusion. This short exact sequence defines  a generator of the
$A$--module $\Ext^1_A(\kk, A) \cong \kk$.
\item In the latter case, let $S$ be any other proper overring of $A$. Then $S$ contains $R$. In other words,
$R$ is the \emph{minimal overring} of the curve singularity $A$.
\item Let $M$ be a Cohen--Macaulay $A$--module without free direct summands. Then there exists
a Cohen--Macaulay $R$--module $N$ such that $M = \II(N)$.
\end{itemize}
\end{theorem}

\begin{remark}\label{R:RejectionLemma} Theorem \ref{T:RejectionLemma} gives a precise measure of the representation theoretic difference between the
categories $\CM(A)$ and $\CM(R)$. Namely, an indecomposable Cohen--Macaulay $A$--module $M$ is either regular or is the restriction of an indecomposable Cohen--Macaulay $R$--module. In more concrete terms, assume
that  $M = \bigl\langle w_1, \dots, w_t \bigr\rangle_A \subset Q^n$ contains no free direct summands (according
to Lemma \ref{L:CMsubmFree}, any Cohen--Macaulay $A$--module admits such embedding).
Then $M = \bigl\langle w_1, \dots, w_t \bigr\rangle_R$.
\end{remark}

\begin{proposition}\label{P:genrestrict}
In the situation of Theorem \ref{T:RejectionLemma}, assume that $N = \bigl\langle w_1, \dots, w_t\bigr\rangle_R \subset Q^n$
is an indecomposable Cohen--Macaulay $R$--module. Then either $N \cong R$  or $N = \bigl\langle w_1, \dots, w_t\bigr\rangle_A$.
\end{proposition}

\begin{proof}
Pose $M := \bigl\langle w_1, \dots, w_t\bigr\rangle_A$. Obviously, we have:
$N = R \cdot M$. If $M$ contains a free direct summand, i.e.~$M \cong M' \oplus A^m$ then
$N = R \cdot M \cong R\cdot M' \oplus R^m$. As $N$ is assumed to be indecomposable, $N \cong R$.
If $N \not\cong R$, then $M$ has no free direct summands. Hence, by Theorem \ref{T:RejectionLemma} and Remark
\ref{R:RejectionLemma} we have:
$R\cdot M = M$.
\end{proof}

\begin{definition}
A Cohen--Macaulay $A$--module $M$ is \emph{locally free on the punctured spectrum of} $A$ if for any minimal prime ideal $\idp$ in
$A$ the localization $M_\idp$ is free over $A_\idp$.
\end{definition}

\begin{remark}\label{R:locfree}
According to Lemma \ref{L:RingofFractions}, a  Cohen--Macaulay $A$--module $M$ is locally free on the punctured spectrum if and only if its rational envelope $Q(M)$ is projective over $Q$.
\end{remark}

\noindent
In what follows, $\CM^{\mathsf{lf}}(A)$ denotes the category
of Cohen--Macaulay $A$--modules which are locally free on the punctured spectrum.

\subsection{Category of triples}

Let $(\rA, \idm)$ be a Cohen--Macaulay  curve singularity, $\rR$ an overring of $\rA$ and
$I = \Ann_\rA(\rR/\rA)$ the corresponding   \emph{conductor ideal}. The next result is straightforward, see for example \cite[Lemma 12.1]{BDNonIsol}.

\begin{lemma} The following statements are true.
\begin{itemize}
\item $I = I \rA = I\rR$. In other words, $I$ is an ideal both in $\rA$ and in $\rR$. Moreover, $I$ is the biggest ideal having this property.
\item The rings $\bar{\rA} = \rA/I$ and $\bar{\rR} = \rR/I$ are Artinian.
\end{itemize}
\end{lemma}

\noindent
For a  Cohen--Macaulay $A$--module $M$  we denote
\begin{itemize}
\item $\tilde{M} := \rR \boxtimes_\rA M \in \CM(\rR)$.
\item $\bar{M} := \bar{\rA} \otimes_{{\rA}} M \in \mod(\bar{\rA})$.
\item $\check{M} := \bar{\rR} \otimes_{\rR} \tilde{M} \in \mod(\bar{\rR})$.
\end{itemize}

\noindent
Then the following result is true, see \cite[Lemma 12.2]{BDNonIsol}.
\begin{lemma}
The canonical map $\theta_M \!: \bar{\rR} \otimes_{\bar{\rA}} \bar{M} \lar \check{M}$ is surjective and its
adjoint map $\tilde{\theta}_M \!: \bar{M} \lar \check{M}$ is injective.
\end{lemma}

\begin{definition}\label{D:triples-on-curves}
Consider the following
\emph{category of triples} $\Tri(\rA)$. Its objects
are  triples $(\mM, V, \theta)$, where
\begin{itemize}
\item $\mM$
is a maximal Cohen--Macaulay $\rR$--module,
\item $V$ is a Noetherian
$\bar\rA$--module,
\item
$\theta\!:  \bar\rR \otimes_{\bar\rA} V \to
\bar\rR \otimes_\rR \mM$ is an epimorphism
of $\bar\rR$--modules such that the adjoint morphism
of $\bar\rA$--modules
$\tilde\theta \!: V \to  \bar\rR \otimes_{\bar\rA} V
\stackrel{\theta}\lar \bar\rR \otimes_\rR \mM $
is a monomorphism.
\end{itemize}
A morphism between two triples $(\mM, V, \theta)$
and $(\mM', V', \theta')$ is given by a pair $(\psi, \varphi)$, where
\begin{itemize}
\item $\psi\!: \mM \to \mM'$ is a morphism of $\rR$--modules and
\item $\varphi\!: V \to V'$ is a morphism of $\bar\rA$--modules
\end{itemize}
such that
the following diagram of $\bar\rR$--modules is commutative:
$$
\xymatrix
{
\bar\rR \otimes_{\bar\rA} V \ar[rr]^-{\theta} \ar[d]_{\mathbbm{1} \otimes \varphi} & &
\bar\rR \otimes_\rR \mM \ar[d]^{\mathbbm{1} \otimes \psi}\\
\bar\rR \otimes_{\bar\rA} V' \ar[rr]_-{\theta'} & &
\bar\rR \otimes_\rR \mM'
}
$$
\end{definition}

\noindent
Definition  \ref{D:triples-on-curves} is motivated  by the following
theorem, see \cite[Theorem 12.5]{BDNonIsol}.

\begin{theorem}\label{T:BDdimOne}
The functor
$
\mathbb{F}\!: \CM(\rA) \lar  \Tri(\rA)
$
mapping a maximal Cohen--Macaulay module $M$ to the triple
$\bigl(\tilde{M}, \bar{M},
\theta_M\bigr)$,  is well--defined and is an equivalence of categories.
A quasi--inverse functor $\GG\!: \Tri(\rA) \lar \CM(\rA)$ is defined as follows.
Let $\Tau = (N, V, \theta)$ be an object of $\Tri(\rA)$.  Then $M' = \GG(\Tau) := \pi^{-1}\bigl(\mathrm{Im}(\tilde\theta)\bigr)\subseteq N$,
where $\pi\!: N \to \bar{N}:=N/IN$ is the canonical projection. In other words, we have the following  commutative diagram
\begin{equation}\label{E:FunctorG}
\begin{array}{c}
\xymatrix
{ 0 \ar[r] & I \mM  \ar[r] \ar@{=}[d] & M' \ar[r] \ar[d]   & V
\ar[d]^-{\tilde{\theta}} \ar[r] & 0 \\
0 \ar[r] & I \mM  \ar[r] & \mM
\ar[r]^{\pi} & \bar{\mM}
\ar[r] & 0
}
\end{array}
\end{equation}
in  the category of $\rA$--modules.
\end{theorem}

\noindent
In many cases, Theorem \ref{T:BDdimOne} provides  an efficient tool to reduce the classification of indecomposable objects of $\CM(\rA)$ to a certain problem of linear algebra (a matrix problem).

\begin{remark}
There are several variations of the construction appearing in Theorem \ref{T:BDdimOne}, see
\cite[Appendix A]{BDNonIsol} for an account  of them.
\end{remark}

\subsection{\texorpdfstring{Cohen--Macaulay modules over simple curve singularities of type $\CA$~~ }{CM modules over simple curve singularities of type A}}

\noindent Let $\kk$ be an algebraically closed field.
For simplicity, let us  additionally assume that $\mathsf{char}(\kk) \ne 2$, see however Remark \ref{R:remarkchar2}.
For any $m \in \mathbb{N}$,  denote
\begin{equation}\label{E:singA}
S = \CA_m := \kk\llbracket x, u\rrbracket/(x^{m+1} - u^2)
\end{equation}  the corresponding simple curve
singularity of type $\CA_m$. The following is essentially due to Bass \cite{Bass}, see also  \cite{LeuschkeWiegand, Yoshino}.

\begin{theorem}\label{T:CMAn}
The indecomposable Cohen--Macaulay $S$--modules have the following description.
\begin{itemize}
\item Assume $m = 2n$, $n \in \NN$. For any $1 \le i \le n$ consider the ideal
$X_i := (x^i, u)$. Then $X_0 = (1) = S, X_1, \dots, X_n$ is the complete list of indecomposable objects of $\CM(S)$. Moreover, the Auslander--Reiten quiver of $\CM(S)$ has the form
\begin{equation}\label{E:AR1}
\xymatrix{ X_0 \ar@<0.75ex>[r]^-{\cdot x} & X_1 \ar@<0.75ex>[r]^-{\cdot x} \ar@<0.5ex>[l]^-{\iota} & X_2 \ar@<0.75ex>[r]^-{\cdot x} \ar@<0.5ex>[l]^-{\iota} & \cdots \ar@<0.75ex>[r]^-{\cdot x} \ar@<0.5ex>[l]^-{\iota} & X_n \ar@<0.5ex>[l]^-{\iota}  \ar@(ur,dr)[]^-{\pi} }
\end{equation}
Here, $\imath$ denotes the inclusion of ideals and  $x\cdot$ is the multiplication by $x$. The endomorphism
$\pi \in \End_S(X_n)$ is defined as follows: $\pi(x^n) = u$ and $\pi(u) = x^{n+1}$.
\item Assume $m = 2n +1$, $n \in \NN_0$. Again, for any  $1 \le i \le n$ consider
$X_i := (x^i, u) \subset S$. Additionally, denote $X^{\pm}_{n+1}:= (x^{n+1} \pm u)$. Then the indecomposable
Cohen--Macaulay $S$--modules are $X_0 = (1) = S, X_1, \dots, X_n, X_{n+1}^+$ and $X_{n+1}^-$. Moreover, the Auslander--Reiten quiver of $\CM(S)$ is in this case
\begin{equation}\label{E:AR2}
\begin{array}{c}
 \xymatrix{ &&&&& X_{n+1}^+ \ar@<0.25ex>[dl]^-{\iota^+} \\
 X_0 \ar@<0.75ex>[r]^-{\cdot x} & X_1 \ar@<0.75ex>[r]^-{\cdot x} \ar@<0.5ex>[l]^-{\iota} & X_2 \ar@<0.75ex>[r]^-{\cdot x} \ar@<0.5ex>[l]^-{\iota} & \cdots \ar@<0.75ex>[r]^-{\cdot x} \ar@<0.5ex>[l]^-{\iota} & X_n \ar@<0.5ex>[l]^-{\iota}  \ar@<1ex>[ur]^-{\pi^+} \ar@<0.5ex>[dr]^-{\pi^-} & \\
&&&&& X_{n+1}^- \ar@<0.75ex>[ul]^-{\iota^-} \\}
\end{array}
\end{equation}
Here, $\iota$ and $\iota^\pm$ denote inclusions of ideals, $x \cdot$ stands for multiplication by $x$. The maps
$\pi^\pm\!: X_n \lar X_{n+1}^\pm$ are defined as follows: $\pi^\pm(x^n) = (x^{n+1} \pm u)$ and $\pi^\pm(u) = x(x^{n+1} \pm u)$.
\end{itemize}
\end{theorem}

\begin{remark}\label{R:remarkchar2}
In the case $\mathsf{char}(\kk) = 2$ there are the following subtleties in defining simple curve singularities of type
 $\CA_m$.
 \begin{itemize}
 \item For $m = 2n +1$ one should take the ring  $\CA_{2n+1} = \kk\llbracket x, u\rrbracket/(u(u- x^{n+1}))$ (the ring defined
 by (\ref{E:singA}) is no longer  reduced!). In this case, one should pose $X_{n+1}^+ := (u)$ and $X_{n+1}^- := (u - x^{n+1})$.
 Then the indecomposable Cohen--Macaulay modules are $X_0, \dots, X_{n}, X_{n+1}^\pm$, where $X_i$ has the same definition
 as in the case $\mathsf{char}(\kk) \ne 2$ for $0 \le i \le n$.

 \item For $m = 2n$ there are more simple singularities than in the case
 $\mathsf{char}(\kk) \ne 2$. Namely, for  $1 \le s \le n-1$ consider the ring $\CA_{2n}^{s} = \kk\llbracket x, u\rrbracket/(u^2 + x^{2n+1} + ux^{n+s})$. Then $\CA_{2n}^{s} \not\cong \CA_{2n}^{t}$ for any $1 \le s \ne t \le n-1$. Moreover,
 $\CA_{2n}^{s} \not\cong \CA_{2n}$ for any $1 \le s \le n-1$. However, the description of indecomposable Cohen--Macaulay modules
 over $\CA_{2n}^s$ is essentially the same as over $\CA_{2n}$, see \cite{Bass} and \cite{KiyekSteinke}.  In particular, the Auslander--Reiten quivers of $\CA_{2n}^{s}$ and  $\CA_{2n}$ coincide.
 \end{itemize}
\end{remark}

\noindent
The following result is due to Buchweitz, Greuel and Schreyer \cite[Section 4.1]{BGS}.

\begin{theorem}\label{T:CMAinfty}
For an algebraically closed field $\kk$ (of arbitrary characteristic) let $S = \CA_\infty := \kk\llbracket x, u\rrbracket/(u^2)$.
Then the indecomposable Cohen--Macaulay $S$--modules are $X_0$, $X_1$, $\dots$, $X_\infty$, where $X_0 = (1) = S, X_\infty = (u)$ and
$X_i = (x^i, u)$ for $i \in \mathbb{N}$. In particular, $X_\infty$ is the only indecomposable Cohen--Macaulay $S$--module which is not locally free on the punctured spectrum of $S$.  The Auslander--Reiten quiver of the category $\CM^{\mathsf{lf}}(S)$ has the form
\begin{equation}\label{E:ARQAinfty1}
\xymatrix{ X_0 \ar@<0.75ex>[r]^-{\cdot x} & X_1 \ar@<0.75ex>[r]^-{\cdot x} \ar@<0.5ex>[l]^-{\iota} &  \cdots \ar@<0.75ex>[r]^-{\cdot x} \ar@<0.5ex>[l]^-{\iota} & X_i \ar@<0.5ex>[l]^-{\iota}  \ar@<0.75ex>[r]^-{\cdot x} \ar@<0.5ex>[l]^-{\iota} & \ar@<0.5ex>[l]^-{\iota} \cdots }
\end{equation}
\end{theorem}

\begin{remark}
It is natural to extend the quiver (\ref{E:ARQAinfty1}) with the remaining indecomposable Cohen--Macaulay $S$--module $X_\infty$.
Moreover, for any $i \in \mathbb{N}_0$ denote $\pi_i\!: X_i \lar X_\infty$ the map sending $x^i$ to $u$ and $u$ to $0$.  Of course,
$\pi_{i+1} = x \cdot \pi_i$ for any $i \in \mathbb{N}_0$. The entire structure of the category $\CM(S)$ can be visualized by the diagram:
\begin{equation}\label{E:ARQAinfty2}
\xymatrix{ X_0 \ar@<0.75ex>[r]^-{\cdot x} & X_1 \ar@<0.75ex>[r]^-{\cdot x} \ar@<0.5ex>[l]^-{\iota} &  \cdots \ar@<0.75ex>[r]^-{\cdot x} \ar@<0.5ex>[l]^-{\iota} & X_i \ar@<0.5ex>[l]^-{\iota}  \ar@<0.75ex>[r]^-{\cdot x} \ar@<0.5ex>[l]^-{\iota} & \ar@<0.5ex>[l]^-{\iota} \ar@<0.75ex>[r]^-{\pi} \cdots & X_\infty \ar@<0.5ex>[l]^-{\iota}}
\end{equation}
\end{remark}

\begin{definition}\label{D:CMstabilized}
Let $(A, \idm)$ be a Cohen--Macaulay curve singularity. Consider the category $\overline{\CM}(A)$ defined as follows:
\begin{itemize}
\item $\Ob\bigl(\overline{\CM}(A)\bigr) = \Ob\bigl(\CM(A)\bigr)$.
\item For $M, N \in \Ob\bigl(\overline{\CM}(A)\bigr)$ we set
$$
\overline{\Hom}_A(M, N) = \mathrm{Im}\Bigl(\Hom_A(M, N) \lar \Hom_{\kk}\bigl(M \otimes_A \kk, N\otimes_A \kk\bigr)\Bigr).
$$
\item The composition of morphisms in $\overline{\CM}(A)$ is induced by the composition of morphisms in $\CM(A)$.
\end{itemize}
\end{definition}

\noindent
The following result is straightforward.
\begin{lemma}\label{L:CMstabilized} The canonical projection functor $\CM(A) \to \overline{\CM}(A)$ is full and respects isomorphy classes  of objects. Moreover, if
$S$ is  a curve singularity of type $\CA_m$ for some $m \in \mathbb{N}^\ast \cup \{\infty\},$ then $\overline{\CM}(S)$ is equivalent to the additive closure of the path algebra category of the corresponding Auslander--Reiten quivers (\ref{E:AR1}), (\ref{E:AR2}) respectively
(\ref{E:ARQAinfty2}) subject to the following zero relations:
\begin{itemize}
\item $(\cdot x) \circ  \imath = \imath \circ (\cdot x)  = 0$.
\item The inclusion $\imath\!: X_1 \to X_0$ is zero in $\overline{\CM}(S)$.
\item $\begin{cases}
\pi^2 = 0, & \text{\sl if $m$ is even}, \\
\pi^\pm \circ \imath^+ = \pi^\pm \circ \imath^- = \imath^+ \circ \pi^{+} +  \imath^- \circ \pi^{-} = 0, & \text{\sl if $m$ is odd}, \\
\pi \circ \iota = 0, &  \text{\sl if $m = \infty$}.
\end{cases}$
\end{itemize}
\end{lemma}

\section{\texorpdfstring{Tameness of  $\CM(\RP_{\infty q})$}{Tameness of Cohen-Macaulay modules over non-reduced singularities of type P}}

\noindent
Let $\kk$ be an algebraically closed field such that $\mathsf{char}(\kk) \ne 2$  and $p, q \in \mathbb{N}_{\ge 2}$. Consider the curve singularity
\begin{align}\label{E:Ppq}
\RP_{pq} &:= \kk\llbracket x, y, z\rrbracket/(xy, x^p + y^q - z^2).
\intertext{By a result of Drozd and Greuel
\cite[Section 3]{DG}, the category $\CM(\RP_{pq})$ is representation tame.  For any $q \in  \mathbb{N}_{\ge 2}$
consider the limiting non--reduced singularity}
\label{E:Pinftyq}
\RP_{\infty q} &:= \kk\llbracket x, y, z\rrbracket/(xy,  y^q - z^2)
\end{align}
as well as the ``largest degeneration''  $\RP_{\infty\infty} := \kk\llbracket x, y, z\rrbracket/(xy,  z^2)$ of the family
(\ref{E:Ppq}).

\medskip
\noindent
The first major  result of this article is the following.

\begin{theorem}\label{T:main1}
The non--reduced curve singularities $\RP_{\infty q}$  have tame Cohen--Macaulay representation type for any
 $q \in \mathbb{N}_{\ge 2} \cup \{\infty\}$.
\end{theorem}

\begin{proof}
\noindent
1.~Since $\RP := \RP_{\infty q}$ is a complete intersection, it is Gorenstein. Let $\RA = \End_{\RP}(\idm)$ be the minimal overring of $\RP$. The first statement of    Theorem \ref{T:RejectionLemma} implies that
\begin{equation}\label{E:minoverring}
\RA   \cong \kk\llbracket x, y, u, v\rrbracket/(xy, yu, uv, vx, u^2, y^q- v^2),
\end{equation}
where the canonical inclusion $\imath\!: \RP \lar \RA$ maps $x$ to $x$, $y$ to $y$ and $z$ to $u+v$ (in what follows,
the generator $y^q - v^2$ has to be replaced by $v^2$ for $q = \infty$).
Note that $u = \frac{\displaystyle xz}{\displaystyle x+y}$ and $v = \frac{\displaystyle yz}{\displaystyle x+y}$ if we regard $\RA$ as a subring of the total ring of fractions
$\RQ = \RQ(\RP)$.  According to Theorem  \ref{T:RejectionLemma}, any non--regular indecomposable Cohen--Macaulay $\RP$--module is a restriction of some indecomposable Cohen--Macaulay $\RA$--module. Thus, it has to be shown that the category
$\CM(\RA)$ has tame representation type.

\medskip
\noindent
2.~ Next, note that $\RR = \kk\llbracket x,u\rrbracket/(u^2) \times \kk\llbracket y, v\rrbracket/(y^q - v^2)$ is an overring of $\RA$.
Indeed, we have an inclusion $\RR  \subset \RQ$, where both idempotents of $\RA$ can be expressed as follows:
$$
e_x := (1, 0) = \frac{x}{x+y} \quad \mbox{\rm and} \quad
e_y := (0, 1) = \frac{y}{x+y}.
$$
The reason  to pass from $\RP$ to its overring $\RA$ is explained by the following observation:
the conductor ideal $I := \mathsf{ann}_{\RA}(\RR/\RA)$ coincides with the maximal ideal $\idm = (x, y, u, v)_\RA$.
Hence,  $\bar{\RA}:= \RA/I \cong  \kk$ and $\bar{\RR} := \RR/I \cong \kk_x \times \kk_y = \kk \times \kk$. Under this identification, the canonical inclusion
$\bar{\RA} \to \bar{\RR}$ is identified with the diagonal embedding.

\noindent According to Theorem \ref{T:BDdimOne}, the category $\CM(\RA)$ is equivalent to the category of triples $\Tri(\RA)$. Thus, we have to show tameness of $\Tri(\RA)$. Let $\Tau = (N, V, \theta)$ be a triple of $\RA$. Then the following facts are true.
\begin{itemize}
\item Since $V$ is just a module over $\bar{\RA} \cong \kk$, we have: $V \cong \kk^t$ for some $t \in \mathbb{N}_0$.
\item Denote  $\RR_x = \kk\llbracket x, u\rrbracket/(u^2)$ and $\RR_y= \kk\llbracket y, v\rrbracket/(y^q - v^2)$. Then
$\RR = \RR_x \times \RR_y$ and $N \cong N_x \oplus N_y$, where $N_x \in \CM(\RR_x)$ and $N_y \in \CM(\RR_y)$.
According to Theorem \ref{T:CMAn} and Theorem \ref{T:CMAinfty}, the Cohen--Macaulay modules
$N_x$ and $N_y$ split  into a direct sum of ideals
\begin{itemize}
\item $X_0 = (e_x)_\RR \cong \RR_x$, $X_i = (x^i, u)_\RR$, for $i \in \mathbb{N}^\ast$ and $X_\infty  = (u)$.
\item $Y_0 = (e_y)_\RR \cong \RR_y$, $Y_j = (y^j, v)$ for $1 \le j \le  s-1$ and
$Y_s^\pm = (y^{s} \pm v)_\RR$
if $q=2s$ is even, respectively  $Y_s = (y^s, v)$ if $q=2s+1$ is odd.
    \end{itemize}
\item We have:  $\bar{N}_x = N_x/IN_x \cong \kk_x^{m}$ and    $\bar{N}_y = N_y/IN_y \cong \kk_y^{n}$ for some $m, n \in \NN_0$. In what follows, we choose bases of $\bar{N}_x$ and $\bar{N}_y$ induced
      by the distinguished generators of the ideals which occur in a direct sum decomposition of $N_x$ and $N_y$.
      Thus, the gluing map $\theta\!: \bar{\RR} \otimes_{\bar{\RA}} V \lar N/IN$ is given by a pair of matrices $(\Theta_x, \Theta_y) \in \Mat_{m \times t}(\kk) \times \Mat_{n \times t}(\kk)$.
\item The condition that the morphism of $\bar{\RR}$--modules $\theta$ is surjective just means  that both matrices $\Theta_x$ and $\Theta_y$ have full row rank.
The condition that $\tilde\theta$ is injective is equivalent to say that the matrix
$\tilde\Theta:= \Bigl(\begin{smallmatrix}\displaystyle \Theta_x \\ \displaystyle \Theta_y\end{smallmatrix}\Bigr)$ has full column rank.
\end{itemize}
\noindent
3.~
Let us now describe the matrix problem underlying a description of isomorphy classes of objects of $\Tri(\RA)$.
 If two triples $\Tau = (N, V, \theta)$ and $\Tau' = (N', V', \theta')$ are isomorphic, then $N \cong N'$ and $V \cong V'$.
Hence, we may without loss of generality pose:
\begin{itemize}
\item $N'  = N = N_x \oplus N_y = \bigl(\bigoplus_{i} X_i^{\oplus m_i}\bigr) \oplus \bigl(\bigoplus_{j} Y_j^{\oplus n_j}\bigr)$ for some $m_i, n_j \in \NN_0$.
\item $V' = V = \kk^t$ for certain $t \in \NN_0$.
\end{itemize}
Then the following is true: we have an isomorphism
$$
\bigl(N, V, (\Theta_x, \Theta_y)\bigr) \cong \bigl(N, V, (\Theta'_x, \Theta'_y)\bigr)
$$
in the category $\Tri(\RA)$ if any only if there exist automorphisms $\Psi_x \in \Aut_{\RR_x}(N_x)$,
$\Psi_y \in \Aut_{\RR_y}(N_y)$ and $\Phi \in \GL_t(\kk)$ such that
\begin{equation}\label{E:transfrule}
\Theta'_x = \bar{\Psi}_x \Theta_x \Phi^{-1} \quad
\mbox{and} \quad \Theta'_y = \bar{\Psi}_y \Theta_y \Phi^{-1},
\end{equation}
where $\bar{\Psi}_x \in \GL_m(\kk)$ (respectively $\bar{\Psi}_y \in \GL_n(\kk)$) is the induced automorphisms of
$\bar{N}_x \cong \kk^m$ (respectively  $\bar{N}_y \cong \kk^n$).  The transformation rule (\ref{E:transfrule}) leads to the following problem of linear algebra (a matrix problem).
\begin{itemize}
\item We have two matrices $\Theta_x$ and $\Theta_y$ over $\kk$ with the same number of columns. The number of rows
of $\Theta_x$ and $\Theta_y$ can be different. In particular, it can be zero for one of these matrices.

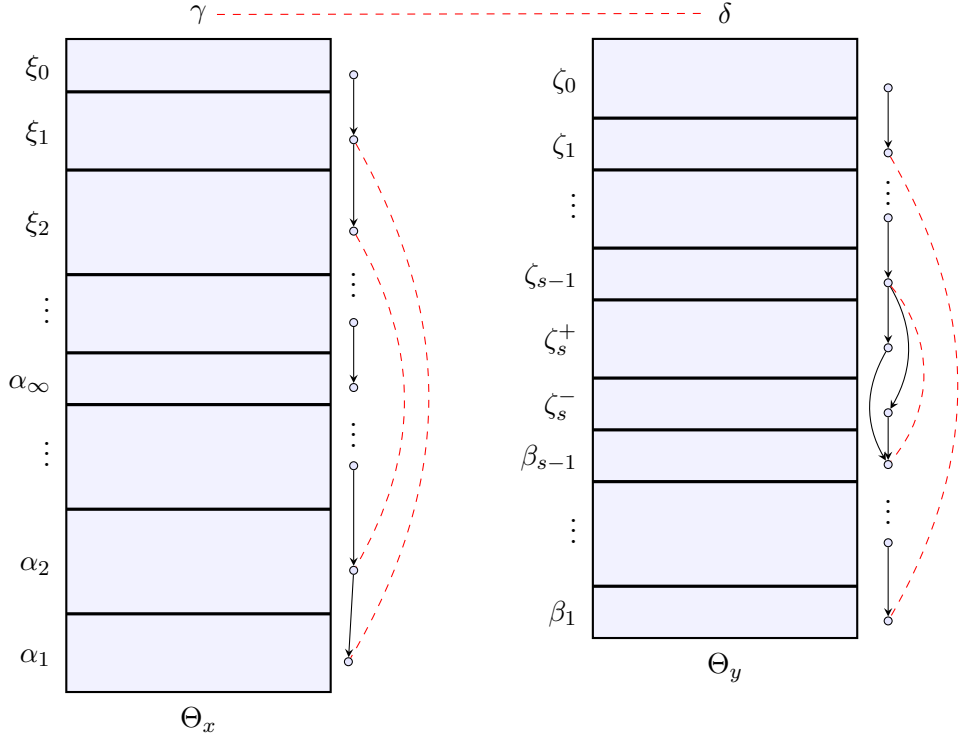
\begin{figure}
\begin{center}
\begin{tikzpicture}[dot/.style={fill=blue!10,circle,draw, inner sep=1pt, minimum size=3pt}]


\matrix[tbl5,text width=100pt, minimum height=20pt,
row 2/.style={minimum height=30pt},
row 3/.style={minimum height=40pt},
row 4/.style={minimum height=30pt},
row 6/.style={minimum height=40pt},
row 7/.style={minimum height=40pt},
row 8/.style={minimum height=30pt},
name=table] at(0,0)
{
~\\
~\\
~\\
~\\
~\\
~\\
~\\
~\\
};

\sfrm{table}{8}{1};
\foreach \i in {1,..., 7}
{
\hsline{table}{\i}{1};
\node[dot,base right=7pt of table-\i-1](n\i){};
}

\node[dot,base right=5pt of table-8-1](n8){};

\draw[-stealth](n1) to (n2); \draw[-stealth](n2) to (n3);
\path (n3) to node{$\vdots$} (n4);
\draw[-stealth](n4) to (n5);
\path (n5) to node{$\vdots$} (n6);
\draw[-stealth](n6) to (n7); \draw[-stealth](n7) to (n8);

 \draw[dashed, red] (n2) to[bend left] (n8); \draw[dashed, red] (n3) to[bend left] (n7);

\node[above=2pt of table-1-1.north](g){$\gamma$};
\node[below=2pt of table-8-1.south]{$\Theta_x$};

\node[base left=2pt of table-1-1]{{$\xi_{0}$}};
\node[base left=2pt of table-2-1]{$\xi_{1}$};
\node[base left=2pt of table-3-1]{$\xi_{2}$};
\node[base left=2pt of table-4-1]{$\vdots$};
\node[base left=2pt of table-5-1]{$\alpha_{\infty}$};
\node[base left=2pt of table-6-1]{$\vdots$};
\node[base left=2pt of table-7-1]{$\alpha_{2}$};
\node[base left=2pt of table-8-1]{$\alpha_{1}$};
%

\matrix[tbl5,text width=100pt, minimum height=20pt,
row 1/.style={minimum height=30pt},
row 3/.style={minimum height=30pt},
row 5/.style={minimum height=30pt},
row 8/.style={minimum height=40pt},
name=table] at(7, 0.36)
{
~\\
~\\
~\\
~\\
~\\
~\\
~\\
~\\
~\\
};

\sfrm{table}{9}{1};
\foreach \i in {1,..., 8}
{
\hsline{table}{\i}{1};
\node[dot,base right=10pt of table-\i-1](n\i){};
}

\node[dot,base right=10pt of table-9-1](n9){};

\draw[-stealth](n1) to (n2);
\path (n2) to node{$\vdots$} (n3);
\draw[-stealth](n3) to (n4); \draw[-stealth](n4) to (n5);
\draw[-stealth](n6) to (n7); \path (n7) to node{$\vdots$} (n8);
\draw[-stealth](n8) to (n9);
 \draw[-stealth] (n4) to[bend left] (n6); \draw[-stealth] (n5) to[bend right] (n7);

 \draw[dashed, red] (n2) to[bend left] (n9);
 \draw[dashed, red] (n4) to[bend left=40pt] (n7);

\node[above=2pt of table-1-1.north](d){$\delta$};
\node[below=2pt of table-9-1.south]{$\Theta_y$};

\draw[dashed, red] (g) to (d);
\node[base left=2pt of table-1-1]{{$\zeta_{0}$}};
\node[base left=2pt of table-2-1]{$\zeta_{1}$};
\node[base left=2pt of table-3-1]{$\vdots$};
\node[base left=2pt of table-4-1]{$\zeta_{s-1}$};
\node[base left=2pt of table-5-1]{$\zeta_{s}^+$};
\node[base left=2pt of table-6-1]{$\zeta_{s}^-$};
\node[base left=2pt of table-7-1]{$\beta_{s-1}$};
\node[base left=2pt of table-8-1]{$\vdots$};
\node[base left=2pt of table-9-1]{$\beta_{1}$};
\end{tikzpicture}
\end{center}
\caption{Matrix problem for the case $q=2s$}
\label{Fig:even}
\end{figure}

\item Rows of $\Theta_x$ are divided into horizontal blocks indexed by elements of the
linearly  ordered set
 $$\qquad \dE_x = \bigl\{ \
 \xi_0 < \xi_1 < \dots < \xi_i < \dots < \alpha_\infty < \dots < \alpha_i < \dots < \alpha_1 \ \bigr\}.$$
 The role of the ordering $<$ will be explained below.
 \item
     The block labeled by $\xi_0$ has $m_0$ rows, the block labeled by $\alpha_\infty$ has $m_\infty$ rows.
     The blocks labeled by $\xi_i$ and $\alpha_i$ both have $m_i$ rows. Thus, the matrix $\Theta_x$ has
    $
    m = m_0 + m_\infty + 2(m_1 + \dots +  m_i + \dots)
    $ rows.
\item The row division of $\Theta_y$ depends on the parity of the parameter $q$.
\begin{itemize}
\item For $q = \infty$ the horizontal blocks of $\Theta_y$ are marked with the symbols of the linearly ordered set
$$ \qquad \qquad
\dE_y = \dE_y^{\infty} = \bigl\{ \ \zeta_0 <  \zeta_1 <  \dots <  \zeta_j <  \dots <  \beta_\infty <  \dots <  \beta_j < \dots <  \beta_1 \ \bigr\},$$
completely analogously as it is done for  $\Theta_x$.
\item For $q=2s+1$, the labels are elements of the linearly ordered set
 $$
 \dE_y = \dE_y^{q} =  \bigl\{ \ \zeta_0 < \zeta_1 < \dots < \zeta_s < \beta_s < \dots < \beta_1 \ \bigr\}. \qquad \qquad
 $$
For any $1 \le j \le s$, the number of rows in blocks marked by $\zeta_j$ and $\beta_j$ is the same (and equal to $n_j$).
\item For $q=2s$, the labels are elements of an (only partially!) ordered set
$$ \qquad
\dE_y = \dE_y^{q} = \bigl\{ \
\zeta_0 < \zeta_1 < \dots <  \zeta_{s-1} <  \zeta_s^+ ,   \zeta_s^- <   \beta_{s-1} <  \dots < \beta_1 \ \bigr\}
$$
as shown in Figure \ref{Fig:even}
(the elements $\zeta_s^+$ and $\zeta_s^-$ are incomparable).
Again, the number of rows in blocks $\zeta_j$ and $\beta_j$ is the same for $1 \le j \le s-1$.
\end{itemize}
\item We can perform
any \emph{simultaneous} elementary  transformation of columns of $\Theta_x$ and $\Theta_y$.
\item Transformations of rows of $\Theta_x$ are of three types.
\begin{itemize}
\item We can add any multiple of any row with lower weight to any row with higher weight.
\item For any $i \in \mathbb{N}$ we can perform any \emph{simultaneous} elementary transformation of rows within blocks marked by
$\xi_i$ and $\alpha_i$.
\item We can make any elementary transformation of rows in block $\xi_0$ or $\alpha_\infty$.
\end{itemize}
\item The transformation rules for rows of $\Theta_y$ \textit{depend on the parity} of $q$.
\begin{itemize}
 \item Let us take the case $q$ is even (the most complicated one, see Figure \ref{Fig:even}).
 \begin{itemize}
 \item We can add any multiple of any row with lower weight to any row with higher weight.
 \item For any $1 \le j \le s-1$ we may perform any \emph{simultaneous} elementary transformation of rows within blocks marked by
$\zeta_j$ and $\beta_j$.
\item We can make any (independent)  elementary transformation of rows in the block $\zeta_0$ or $\zeta_s^\pm$.
 \end{itemize}
 \item
For $q = \infty$, the transformation rules for $\Theta_y$ are analagous to those listed above for $\Theta_x$ 
(The matrix problem for this case will be studied in much detail in Subsection \ref{dXofP}).
\item
For odd $q=2s+1$, the transformation rules of $\Theta_y$  are the same as for $q= \infty$. The only difference between these cases lies in the absence of certain symbols in $\dE_y$.
\end{itemize}
\end{itemize}
\noindent
4.~ For any $q \in \mathbb{N}_{\geq 2} \cup \{\infty\}$ the described matrix problem is an example of  \textsl{representations of a bunch of semi--chains}. Tameness of the latter class of problems has been shown by Bondarenko in \cite{bo1}. This implies the tameness of the category of triples $\Tri(\RA)$. Tameness of $\CM(\RA)$ follows from Theorem \ref{T:BDdimOne}.
\end{proof}

\begin{remark} \label{semichains} Consider the following combinatorial data:
\begin{itemize}
\item The index set $I = \{x, y\}$.
\item Let $\dF_x = \{\gamma\}$, $\dF_y = \{\delta\}$, $\dF = \dF_x \cup \dF_y$.
\item Let $\dE_x$ and $\dE_y = \dE_y^q$ be as above, $\dE = \dE_x \cup \dE_y$.
\item In the set $\dX = \dE \cup \dF$ consider the symmetric  relation $\sim$ defined as follows:
\begin{align*}
\gamma \sim \delta, \qquad &\xi_i \sim \alpha_i \quad \text{ for } i \in \mathbb{N}^\ast, \\
&\zeta_j \sim \beta_j \quad \text{ for }
\begin{cases}
1 \le j \le s & \text{if }q=2s \text{ or }q=2s+1, \\
j \in \NN & \text{if }q= \infty.
\end{cases}
\end{align*}
\end{itemize}
The entire data (which is an example of a bunch of semi--chains)  defines a certain \emph{bimodule category} $\Rep(\dX)$, see \cite{DrozdLOMI} or
\cite{BDNonIsol} for more details.
The description of isomorphy classes of objects of $\Rep(\dX)$ reduces   precisely to the matrix problem described above.
In our  particular case, the category $\Rep(\dX)$ admits the following intrinsic description:
it is the \emph{comma category} of the following diagram of categories and functors:
$$
\overline{\CM}(\RR_x) \times \overline{\CM}(\RR_y) \xrightarrow{\mathsf{For}} \mod(\kk \times \kk)
\xleftarrow{(\kk \times \kk) \otimes_\kk \,-\,} \mod(\kk),
$$
where $\overline{\CM}(\RR_x)$ and  $\overline{\CM}(\RR_y)$  have been  defined in Definition \ref{D:CMstabilized} (see also
Lemma \ref{L:CMstabilized} for their explicit description) and
 $\mathsf{For}$ is the forgetful functor. According to Bondarenko  \cite{bo1}, there are the following types of indecomposable objects in $\Rep(\dX)$: bands (continuous series) and (bispecial, special and usual) strings (discrete series).
The precise combinatorics of the discrete series  is rather complicated.
\end{remark}

\begin{definition}
The forgetful functor $\PP\!: \Tri(\RA) \longrightarrow \Rep(\dX)$ assigns
 assigns to a triple $\bigl(N, V, (\Theta_x, \Theta_y)\bigr)$ the pair of partitioned matrices $(\Theta_x, \Theta_y)$. To be more precise, we recall that
\begin{itemize}
\item $N = N_x \oplus N_y = \bigl(\bigoplus_{i} X_i^{\oplus m_i}\bigr) \oplus \bigl(\bigoplus_{j} Y_j^{\oplus n_j}\bigr)$ where $m_i, n_j \in \mathbb{N}_0$.
\item $V = \kk^t$ for some $t \in \mathbb{N}_0$.
\item $\Theta_x\!: V = \kk^t \longrightarrow \bar\RR \otimes_{\RR} N_x \cong \kk^{m}$,\\ $\Theta_y\!: V = \kk^t \longrightarrow \bar\RR \otimes_{\RR} N_y \cong \kk^{n}$ for certain $m, n \in \mathbb{N}_0$.
\end{itemize}
At this point, we have chosen bases for $\bar\RR \otimes_{\RR} N_x$ respectively $\bar\RR \otimes_{\RR} N_y$ which are induced by the distinguished generators of the indecomposable modules over $R_x$ and $R_y$ as in the body of the proof of Theorem \ref{T:main1}.
These generators correspond exactly to the elements of the bunch of semi--chains $\dX$
indexing the horizontal stripes of $\Theta_x$ and $\Theta_y$.
\end{definition}
\begin{remark}
The functor $\PP$ has the following properties.
\begin{enumerate}
\item $\PP$ is additive, full and preserves isomorphism classes of objects.
\item The essential image of $\PP$ consists of all pairs of matrices $(A,B)$ in $\Rep(\dX)$ such that $A$ and $B$ have both full row rank and $\Bigl(\begin{smallmatrix}
\displaystyle A \\ \displaystyle B \end{smallmatrix}\Bigr)$ has full column rank.
\item $\PP$ is not faithful.
\end{enumerate}
\end{remark}

\begin{remark}\label{R:RonPQchar2}
Let $\mathsf{char}(\kk) = 2$. Then the simple curve singularities of type $\CA$ have to be redefined according to Remark \ref{R:remarkchar2}.
It follows that the equation of $\RP_{\infty, 2s}$ should be
$\kk\llbracket x, y, z\rrbracket/(xy, z(y^s -z))$.
Moreover, there are more singularities of type
$\RP_{\infty, 2s+1}$, namely
\begin{align*}
\RP_{\infty, 2s+1}^{t} := \kk \llbracket x,y,z\rrbracket/(xy, y^{2s+1} + y^{s+t} z - z^2 ), \qquad 1\leq t \leq s-1.
\end{align*}
Nevertheless, they are all tame and the proof of Theorem \ref{T:main1} applies literally to this case as well.
\end{remark}

\begin{remark}\label{R:proofoftameness}
For any $q \in \mathbb{N} \cup \{\infty\}$ consider the hypersurface singularity
$$\RT = \RT_{\infty, q+2} = \kk\llbracket a, b\rrbracket/\bigl(b^2(a^2-b^{q})\bigr).$$
Observe that $\RA$ is an overring of $\RT$ via the embedding
\begin{align*}
\RT \lar \RA,  \qquad a \longmapsto x+v,
\quad b \longmapsto y+u
\end{align*}
where $\RA$ is the ring defined by (\ref{E:minoverring}).
It was shown in \cite[Section 11.1]{BDNonIsol} that
$\CM(\RT)$ has tame representation type (under the additional assumption  $\mathsf{char}(\kk) = 0$). This gives another argument
that $\CM(\RA)$ (and hence $\CM(\RP)$) has either tame or discrete Cohen--Macaulay type. The latter case does also occur:
if  $q=1$, then $\RT_{\infty 3}$ is representation tame whereas
$$
\RP_{\infty 1} = \kk\llbracket x, y, z\rrbracket/(xy, y-z^2) \cong \kk\llbracket x, z\rrbracket/(xz^2) =: \RD_\infty
$$
is  representation discrete \cite{BGS}.
\end{remark}

\section{\texorpdfstring{Cohen--Macaulay modules over $\RP_{\infty\infty}$ and $\RT_{\infty\infty}$}{Cohen--Macaulay modules over the maximal degenerations of type P and T}} \label{S:explizit}
In this section we shall explain that the technique of matrix problems, introduced in the course of the proof of Theorem \ref{T:main1}, leads to a completely explicit description of indecomposable Cohen--Macaulay modules over
$\RP = \RP_{\infty\infty} = \kk\llbracket x,y,z\rrbracket/(xy, z^2)$. Although $\RP$ is the ``maximal degeneration'' of the family
(\ref{E:Ppq}), the combinatorics of the indecomposable objects in $\CM(\RP)$ are  more transparent
 than for the less degenerate singularities  $\RP_{\infty, 2r}$.
The reason is that the underlying matrix problem has  the type \textsl{representations of bunches of chains} and not of \textsl{semi--chains} as for $\RP_{\infty, 2r}$.
  Another motivation to study Cohen--Macaulay modules over $\RP$ is that it allows to construct interesting examples of Cohen--Macaulay modules over the hypersurface singularity
$\RT = \RT_{\infty \infty} = \kk\llbracket a, b\rrbracket/(a^2 b^2)$.
Moreover, the explicit classification of indecomposable Cohen--Macaulay modules over $\RP$ yields a classification for any curve singularity of type $\RP_{\infty, 2s+1}$ or $\RP_{2r+1,2s+1}$.

\medskip
\noindent
Until the end of  this section we keep   the following notation:
\begin{itemize}
\item $\RA = \kk\llbracket x, y, u, v\rrbracket/(xy, yu, uv, vx, u^2, v^2)$ is the minimal overring of $\RP$.
The embedding $\RP \lar \RA$ sends $z$ to $u + v$.
\item Let $\RR = \RR_x \times \RR_y = \kk\llbracket x, u\rrbracket/(u^2) \times \kk\llbracket y, v\rrbracket/(v^2)$.
\item For any $l \in \NN$ we denote $X_l = (u, x^l)_\RR$ and $Y_l = (v, y^l)_\RR$. Next, we pose $X_0 = \RR_x$, $Y_0 = \RR_y$,
$X_\infty = (u)_\RR$ and $Y_\infty = (v)_\RR$.
\item Let
$\RQ = \kk\llbrace x\rrbrace[u]/(u^2) \times \kk\llbrace y\rrbrace[v]/(v^2)$.
\item Denote $\idm = (x, y, u, v)_\RA$. Recall that  $\idm = \mathsf{ann}_\RA(\RR/\RA) = \mathsf{rad}(\RR)$.
\end{itemize}
Observe  that $\RQ$ is the common total ring of fractions of $\RP, \RA$ and $\RR$. In particular, we have the following equalities in $\RQ$:
\begin{equation}
u = \frac{xz}{x+y}, \quad v = \frac{yz}{x+y}, \quad
e_x :=  (1, 0) = \frac{x}{x+y} \quad \mbox{and} \quad  e_y := (0, 1) = \frac{y}{x+y}.
\end{equation}
Next, note that $\RA$ is also an overring of $\RT$. Indeed, we have an injective ring homomorphism
$\jmath\!:  \RT \lar \RA$ given by $a \mapsto x + v$ and $b \mapsto y+u$. It is also not difficult to see that $\jmath$ induces an isomorphism of the total rings of fractions
$\RQ(\RT) \lar \RQ(\RA) = \RQ$. Summing up, we have the following diagram of categories and functors:
\begin{equation}
\begin{array}{c}
\xymatrix@M+1pt{
\CM(\RP) & & & \\
& \CM(\RA) \ar@{}[r]|-{\sim} \ar@/^/[r]^-{\FF} \ar@{_{(}->}[lu]_-{\II} \ar@{^{(}->}[ld]^-{\JJ} & \Tri(\RA) \ar[r]^{\PP} \ar@/^/[l]^-{\GG} & \Rep(\dX) \\
\CM(\RT) & & &
}
\end{array}
\end{equation}

\begin{itemize}
\item $\II$ and $\JJ$ denote restriction functors. According to Proposition \ref{P:overrings}, they are both fully faithful.
Moreover, by Theorem \ref{T:RejectionLemma} $\mathsf{Ind}\bigl(\CM(\RP)\bigr) = \{\RP\} \cup \mathsf{Ind}\bigl(\CM(\RA)\bigr)$.
\item $\FF$ and $\GG$ are  quasi--inverse equivalences of categories from  Theorem \ref{T:BDdimOne}.
\item The functor $\PP$ assigns to a triple $\bigl(N, V, (\Theta_x, \Theta_y)\bigr)$ the pair of matrices $(\Theta_x, \Theta_y)$, whose rows are equipped with some additional ``weights''. $\PP$ preserves isomorphy classes of objects as well as their indecomposability. However, $\PP$ is not essentially surjective because $\Theta_x$ and $\Theta_y$ obey some additional constraints, see (\ref{E:regular}).
\end{itemize}
The goal of this section is to show how one can translate the combinatorics of indecomposable objects of $\Rep(\dX)$ into an
explicit description of indecomposable objects of $\CM(\RP)$ and $\CM(\RT)$.

\subsection{\texorpdfstring{Indecomposable objects of $\Rep(\dX)$ }{Indecomposable objects of Rep(B)0}}\label{dXofP}

According to Theorem \ref{T:main1}, the matrix problem corresponding to a description of isomorphy classes of objects in $\rep(\dX)$ is as follows.

\medskip
\noindent We are given two matrices $\Theta_x$ and $\Theta_y$ as depicted in Figure \ref{fig:MatProb} with entries from an algebraically closed field $\kk$ and the same number of columns.
The rows of $\Theta_x$ and $\Theta_y$ are both divided into horizontal blocks.
Any two horizontal blocks in $\Theta_x$ (respectively $\Theta_y$)   connected by a dotted line  have the same number of rows.

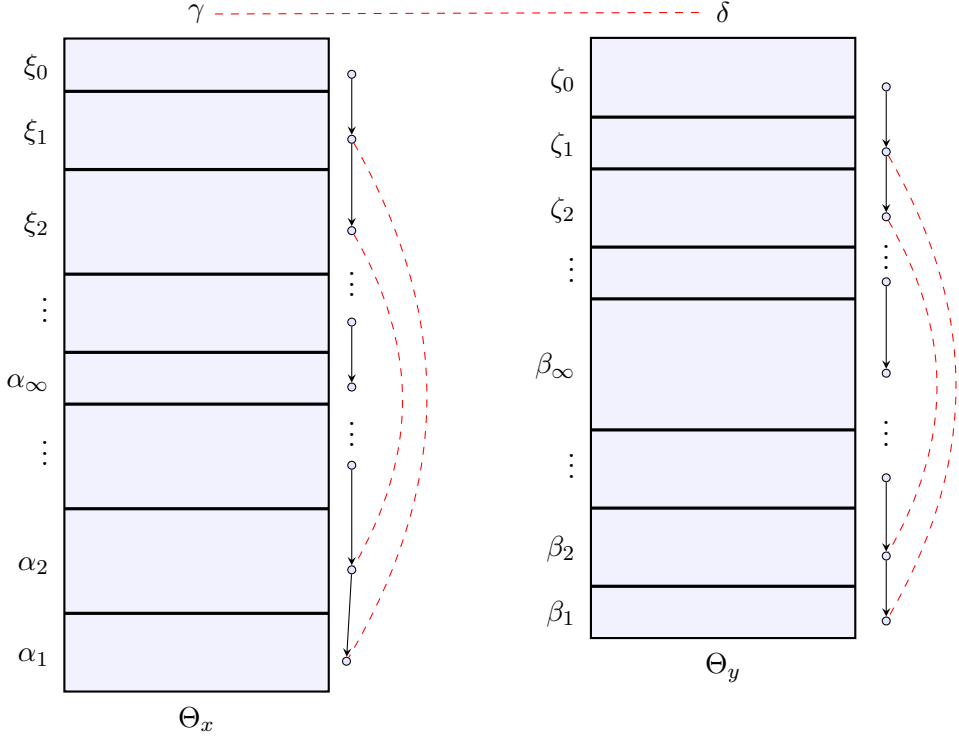
\begin{figure}
\begin{tikzpicture}[dot/.style={fill=blue!10,circle,draw, inner sep=1pt, minimum size=3pt}]


\matrix[tbl5,text width=100pt, minimum height=20pt,
row 2/.style={minimum height=30pt},
row 3/.style={minimum height=40pt},
row 4/.style={minimum height=30pt},
row 6/.style={minimum height=40pt},
row 7/.style={minimum height=40pt},
row 8/.style={minimum height=30pt},
name=table] at(0,0)
{
~\\
~\\
~\\
~\\
~\\
~\\
~\\
~\\
};

\sfrm{table}{8}{1};
\foreach \i in {1,..., 7}
{
\hsline{table}{\i}{1};
\node[dot,base right=7pt of table-\i-1](n\i){};
}

\node[dot,base right=5pt of table-8-1](n8){};

\draw[-stealth](n1) to (n2); \draw[-stealth](n2) to (n3);
\path (n3) to node{$\vdots$} (n4);
\draw[-stealth](n4) to (n5);
\path (n5) to node{$\vdots$} (n6);
\draw[-stealth](n6) to (n7); \draw[-stealth](n7) to (n8);

 \draw[dashed, red] (n2) to[bend left] (n8); \draw[dashed, red] (n3) to[bend left] (n7);

\node[above=2pt of table-1-1.north](g){$\gamma$};
\node[below=2pt of table-8-1.south]{$\Theta_x$};

\node[base left=2pt of table-1-1]{{$\xi_{0}$}};
\node[base left=2pt of table-2-1]{$\xi_{1}$};
\node[base left=2pt of table-3-1]{$\xi_{2}$};
\node[base left=2pt of table-4-1]{$\vdots$};
\node[base left=2pt of table-5-1]{$\alpha_{\infty}$};
\node[base left=2pt of table-6-1]{$\vdots$};
\node[base left=2pt of table-7-1]{$\alpha_{2}$};
\node[base left=2pt of table-8-1]{$\alpha_{1}$};
%

\matrix[tbl5,text width=100pt, minimum height=20pt,
row 1/.style={minimum height=30pt},
row 3/.style={minimum height=30pt},
row 5/.style={minimum height=50pt},
row 6/.style={minimum height=30pt},
row 7/.style={minimum height=30pt},
name=table] at(7, 0.36)
{
~\\
~\\
~\\
~\\
~\\
~\\
~\\
~\\
};

\sfrm{table}{8}{1};
\foreach \i in {1,..., 7}
{
\hsline{table}{\i}{1};
\node[dot,base right=10pt of table-\i-1](n\i){};
}

\node[dot,base right=10pt of table-8-1](n8){};

\draw[-stealth](n1) to (n2);\draw[-stealth](n2) to (n3);
\path (n3) to node{$\vdots$} (n4);
\draw[-stealth](n4) to (n5); \path (n5) to node{$\vdots$} (n6);
\draw[-stealth](n6) to (n7); \draw[-stealth](n7) to (n8);

 \draw[dashed, red] (n2) to[bend left] (n8);
 \draw[dashed, red] (n3) to[bend left] (n7);

\node[above=2pt of table-1-1.north](d){$\delta$};
\node[below=2pt of table-8-1.south]{$\Theta_y$};

\draw[dashed, red] (g) to (d);
\node[base left=2pt of table-1-1]{{$\zeta_{0}$}};
\node[base left=2pt of table-2-1]{$\zeta_{1}$};
\node[base left=2pt of table-3-1]{$\zeta_{2}$};
\node[base left=2pt of table-4-1]{$\vdots$};
\node[base left=2pt of table-5-1]{$\beta_{\infty}$};
\node[base left=2pt of table-6-1]{$\vdots$};
\node[base left=2pt of table-7-1]{$\beta_{2}$};
\node[base left=2pt of table-8-1]{$\beta_{1}$};
\end{tikzpicture}
\caption{Matrix problem for the case $q= \infty$}
\label{fig:MatProb}
\end{figure}

\medskip
\noindent
\textsl{Transformation rules}.
The following transformations of columns and rows of $\Theta_x$ and $\Theta_y$ are admissible:
\begin{itemize}
\item any \emph{simultaneous} elementary  transformation of columns of $\Theta_x$ and $\Theta_y$.
\item addition of  any multiple of any row of $\Theta_x$ (respectively, $\Theta_y$)  with lower weight to any row of $\Theta_x$ (respectively, $\Theta_y$) with higher weight.
\item any \emph{simultaneous} elementary transformation of rows within horizontal blocks  of $\Theta_x$ (respectively, $\Theta_y$) connected by a dotted line.
\item any elementary transformation of rows in the  horizontal block of $\Theta_x$ (respectively, $\Theta_y$) which is not connected to any other block by a dotted line.
\end{itemize}

\noindent Additionally, there are the following  \textsl{regularity constraints} on $\Theta_x$ and $\Theta_y$:
\begin{equation}\label{E:regular}
\begin{array}{l}
\mbox{$\Theta_x$ and $\Theta_y$ have both full row rank}, \\
\mbox{the matrix  $\Bigl(\begin{smallmatrix}\displaystyle \Theta_x \\ \displaystyle \Theta_y\end{smallmatrix}\Bigr)$
 has full column rank}.
\end{array}
\end{equation}
\noindent
In this subsection we are going to apply Bondarenko's result \cite{bo1} to describe the canonical forms of the matrix problem above.

\begin{definition}\label{D:distinguishedbuch} Consider the following data:
\begin{itemize}
\item The index set $I = \{x, y\}$.
\item The set of column symbols $\dF = \dF_x \cup \dF_y$, where $\dF_x = \{\gamma\}$, $\dF_y = \{\delta\}$.
\item The set of row symbols $\dE = \dE_x \cup \dE_y$, where
\begin{align*}
\dE_x &= \bigl\{ \
 \xi_0 < \xi_1 < \dots < \xi_i < \dots < \alpha_\infty < \dots < \alpha_i < \dots < \alpha_1 \ \bigr\}, \\
\dE_y &= \bigl\{ \ \zeta_0 <  \zeta_1 <  \dots <  \zeta_j <  \dots <  \beta_\infty <  \dots <  \beta_j < \dots <  \beta_1 \ \bigr\}.
\end{align*}
\item The set $\dX = \dE \cup \dF$ is equipped with a symmetric relation $\sim$ defined as follows:
\begin{align*} \gamma \sim \delta, \qquad \xi_l \sim \alpha_l \quad \text{and} \quad \zeta_l \sim \beta_l \quad \text{for $l \in \mathbb{N}^\ast$}.
\end{align*}
\end{itemize}
In addition, we introduce another symmetric relation $-$ on the set $\dX$ as follows:
\begin{align*}
\gamma - \epsilon_x\text{ for any }\epsilon_x \in \dE_x \quad \text{ and }\quad \delta - \epsilon_y \text{ for any }
\epsilon_y \in \dE_y.
\end{align*}
\end{definition}

\noindent
The problem to classify the indecomposable objects of the category $\Rep(\dX)$
up to isomorphism
is exactly the matrix problem above without the ``regularity conditions'' (\ref{E:regular}).

\medskip
\noindent
Now, we define \emph{strings} and \emph{bands} of the bunch of chains $\dX$. They describe the invariants of the indecomposable  representations in $\Rep(\dX)$.
\begin{definition}\label{string1} Let $\dX$ be the bunch of chains from Definition \ref{D:distinguishedbuch}.
\begin{enumerate}
\item
A \emph{full word} $w$ of $\dX$ is a sequence
$$w = \chi_1 \rho_1 \chi_2 \rho_2 \ldots \chi_{n-1} \rho_{n-1} \chi_n$$
 of symbols $\chi_k \in \dX$ and relations $\rho_k \in \{\sim, -\}$ subject to the following conditions:
 \begin{itemize}
 \item the relation $\chi_k \rho_k \chi_{k+1}$ holds  in $\dX$ for $1 \leq k \leq n-1$.
 \item the sequence of relations alternates, i.e. $\rho_k \neq \rho_{k+1}$ for $1 \leq k \leq n-2$.
 \item either $\chi_1 \in \{ \xi_0, \alpha_\infty, \zeta_0, \beta_\infty\}$ or $\rho_1$ is $\sim$.
\item either $\chi_n \in \{ \xi_0, \alpha_\infty, \zeta_0, \beta_\infty \}$ or $\rho_{n-1}$ is  $\sim$.
\end{itemize}
\item The \emph{opposite word} of a full word $w$ as above is defined by
$$w^{o} = \chi_{n} \rho_{n-1} \chi_{n-1} \rho_{n-2} \ldots \chi_{2} \rho_{1} \chi_1.$$
\item A \emph{string datum} of $\dX$ is given by any full word $w$.
\item Two string data $w$ and $w'$ are \emph{equivalent} if and only if $w'=w$ or $w' = w^{o}$.
\end{enumerate}
\end{definition}
\begin{definition}\label{band1} Let $\dX$ be the bunch of chains from Definition \ref{D:distinguishedbuch}.
\begin{enumerate}
\item A \emph{cyclic word} $w$ is given by a full word $w=\chi_1 \rho_1 \chi_2 \ldots \chi_{n-1} \rho_{n-1} \chi_n$ and an additional relation $\rho_{n}$ equal to $-$ such that the following additional conditions are satisfied:
\begin{itemize}
\item $\chi_n \rho_n \chi_1$ holds in $\dX$.
\item $\rho_1$ and $\rho_{n-1}$ are equal to $\sim$.
\end{itemize}
\item
The opposite of a cyclic word as above is given by the full word $w^{o}$ together with the additional relation $\rho_n=-$.
\item For any even integer $k \in \mathbb{Z}/ n \mathbb{Z}$ the \emph{$k$-th shift} of $w$ is defined by
$$w^{[k]} = \chi_{k+1} \rho_{k+1} \chi_{k+2} \rho_{k+2} \ldots \chi_{k+n-1} \rho_{k+n-1} \chi_{k+n} \rho_{k+n}, $$
where the indices are considered modulo $n$.
\item A cyclic word $w$ is \emph{periodic} if $w = w^{[k]}$ for some even non-trivial shift
$k\in \bigl(\mathbb{Z} / n \mathbb{Z}\bigr)^\ast$.
\item A \emph{band datum} $(w, m, \lambda)$ of $\dX$ consists of a non--periodic cyclic word $w$, a multiplicity parameter $m \in \NN$ and a continuous parameter $\lambda \in \kk^\ast$.
\item
Two band data $(w,m ,\lambda)$ and $(w',m', \lambda')$ are \emph{equivalent} if and only if both words $w$ and $w'$ have the same length $n$ and
if $(w',m', \lambda')$ is given by
$(w^{o}, m , \lambda)$, $(w^{[4l]},m, \lambda)$ or
$(w^{[4l+2]}, m, \lambda^{-1})$ for some $l \in \mathbb{Z} / n \mathbb{Z}$.
\end{enumerate}
\end{definition}

\noindent
The above  definitions are motivated by the following result of Bondarenko \cite{bo1}.
\begin{theorem}\label{T:BondarenkoClassific}
There is a bijection between the equivalence classes of string and band data of the bunch of chains $\dX$ and the isomorphism classes of indecomposable objects in the category $\Rep(\dX)$.
\end{theorem}

\noindent
Now we  explain   Bondarenko's construction of indecomposable objects in $\Rep(\dX)$ corresponding to
 a string or band datum.

\medskip
\noindent 1.~Let $w$ be a string datum of $\dX$.
The corresponding  object $\Theta(w)$ of $\Rep(\dX)$  is given by a pair of  matrices $\Theta_x(w)$ and $\Theta_y(w)$ defined  as follows:
\begin{enumerate}
\item Let $t$ be the number of times the symbol $\gamma$ (or $\delta$) occurs
as a letter in $w$. Then both matrices $\Theta_x(w)$ and $\Theta_y(w)$ have $t$ columns.
\item For each $\epsilon \in \dE$,
let $m_\epsilon$ be the number of times the symbol $\epsilon$ occurs as a letter in $w$.
Then the horizontal block $\epsilon$ in $\Theta_x(w)$ (respectively  $\Theta_y(w)$)   has
$m_\epsilon$ rows.
\item Next, we assign to every letter $\chi_k$ in $w$ the number of times the letter $\chi_k$ occurred in the subword $\chi_1 \rho_1 \ldots \rho_{k-1} \chi_k$.
In other words, we number every letter in $w$ by the time it occurs in $w$.
\item Every appearance of the relation $-$ in $w$ contributes to a non-zero entry in $\Theta(w)$ in the following way.
Let $\epsilon - \nu$ or $\nu - \epsilon$ be a subword in $w$ such that $\epsilon \in \dE$ and $\nu \in \dF$.
Let $i$ be the occurrence number of $\epsilon$ and $j$ be the occurrence number
of $\nu$.
We fill the entries of $\Theta_x$ respectively $\Theta_y$ according to  the following rule.
\begin{itemize}
\item If $\nu = \gamma$ (respectively $\nu = \delta$), the $(i, j)$-th entry of the $\epsilon$-th horizontal block of $\Theta_x$ (respectively $\Theta_y$) is set to be $1$.
This rule is applied for every relation $-$ in $w$.
\item All remaining entries of $\Theta_x$ and $\Theta_y$ are set to be $0$.
\end{itemize}
\end{enumerate}

\medskip
\noindent
2.~Let $(w,m,\lambda)$ be a band datum.
The corresponding object   $\Theta(w, m, \lambda)$ of $\Rep(\dX)$ is given by a pair of matrices
$\bigl(\Theta_x(w, m, \lambda),\Theta_y(w, m, \lambda)\bigr)$  defined as follows.
\begin{enumerate}
\item First construct  the canonical form  $\Theta(w)$, where $w$ is viewed as a string parameter.
\item Replace any zero entry in $\Theta_x(w)$ (respectively $\Theta_y(w)$)   by the zero matrix of size $m$ and
any identity entry in $\Theta_x(w)$ (respectively $\Theta_y(w)$) by the identity matrix $I$ of size $m$.
\item We consider the last relation $\rho_n=-$ in $w$.
Then $\chi_n \rho_n \chi_1$ in $\dX$ and either $\chi_1$ or $\chi_n \in \dE$.
Replace the zero matrix at the intersection of the first $m$ rows of the horizontal block indexed by $\chi_1$ (respectively $\chi_n$) and its last $m$ columns by the Jordan block $J_m(\lambda)$ with eigenvalue $\lambda$ and size $m$.
\end{enumerate}

\medskip
\begin{remark}\label{R:EssentImage}
Let $\Theta$ be the object of $\Rep(\dX)$ associated to a string or band datum.
It is straightforward to show, that $\Theta$ satisfies the  regularity constraints (\ref{E:regular})
if and only if the following two conditions are satisfied:
\begin{itemize}
\item either $w$ is cyclic or $w$ begins and ends with symbols from $\dF$.
\item $w$ is neither equal to $\gamma \sim \delta$ nor equal to $\delta \sim \gamma$.
\end{itemize}
\end{remark}

\noindent
Next, we give some examples of canonical forms of band and string data.
\begin{example} \label{ex1} Consider a band datum $(w,m,\lambda)$ where $w$ is given by the following cyclic word:
$$
w = 
\underset{\# 1 \ \ \ \#1}{\delta \sim \gamma} -
\underset{\# 1 \ \ \ \#1 }{\xi_{i} \sim \alpha_{i}} -
\underset{\# 2 \ \ \ \# 2}{\gamma \sim \delta} -
\underset{\# 1 \ \ \ \# 1}{\zeta_{j_1} \sim \beta_{j_1}} -
\underset{\# 3 \ \ \ \# 3}{\delta \sim \gamma} -
\underset{\# 2 \ \ \ \# 2}{\alpha_{i} \sim \xi_{i}} -
\underset{\# 4 \ \ \ \# 4}{\gamma \sim \delta} -
\underset{\# 1 \ \ \ \# 1}{\zeta_{j_2} \sim \beta_{j_2}} -
$$
Then the corresponding canonical forms $\bigl(\Theta_x(w,m,\lambda), \Theta_y(w,m,\lambda)\bigr)$ are the following:
\begin{center}
\begin{tikzpicture}

\matrix[tbl5,text width=20pt, minimum height=20pt,  name=table] at(0,0)
{
I&      0&     0&     0\\
0&      0&     0&     I\\
0&      I&     0&     0\\
0&      0&     I&     0\\
};

\sfrm{table}{4}{4};
\hdotline{table}{1}{4};
\hdline{table}{2}{4};
\hdotline{table}{3}{4};
\vdotline{table}{4}{1};
\vdotline{table}{4}{2};
\vdotline{table}{4}{3};

\node[above=2pt of table-1-2.north east]{$\gamma$};
\node[base left=2pt of table-1-1.south west]{{$\xi_{i}$}};
\node[base left=2pt of table-3-1.south west]{$\alpha_{i}$};
%

\matrix[tbl5,text width=20pt, minimum height=20pt,  name=table] at(4,0)
{
0&      I&     0&     0\\
0&      0&     0&     I\\
0&      0&     I&     0\\
J&      0&     0&     0\\
};

\sfrm{table}{4}{4};
\sfrm{table}{4}{4};
\hdline{table}{1}{4};
\hdline{table}{2}{4};
\hdline{table}{3}{4};
\vdotline{table}{4}{1};
\vdotline{table}{4}{2};
\vdotline{table}{4}{3};

\node[above=2pt of table-1-2.north east]{$\delta$};
\node[base right=2pt of table-1-4]{{$\zeta_{j_1}$}};
\node[base right=2pt of table-2-4]{$\zeta_{j_2}$};
\node[base right=2pt of table-3-4]{$\beta_{j_1}$};
\node[base right=2pt of table-4-4]{$\beta_{j_2}$};
\end{tikzpicture}
\end{center}
Here, $I$ is the identity matrix of size $m$ and $J = J_m(\lambda)$ is the Jordan block of size $m$ with eigenvalue $\lambda \in \kk^\ast$.
\end{example}

\begin{example} \label{ex2} Consider the string datum given by the word
$$
w = \delta \sim \gamma - \xi_i \sim \alpha_i - \gamma \sim \delta - \zeta_j \sim \beta_j - \delta \sim \gamma.
$$
Then the corresponding canonical forms $\bigl(\Theta_x(w), \Theta_y(w)\bigr)$ are the following:
\begin{center}
\begin{tikzpicture}

\matrix[tbl5,text width=15pt, minimum height=15pt,  name=table] at(0,0)
{
      1&     0&     0\\
      0&     1&     0\\
};

\sfrm{table}{2}{3};
\node[above=2pt of table-1-2]{$\gamma$};
\node[base left=2pt of table-1-1]{{$\xi_{i}$}};
\node[base left=2pt of table-2-1]{$\alpha_{i}$};

\matrix[tbl5,text width=15pt, minimum height=15pt,  name=table] at(3,0)
{
      0&     1&     0\\
      0&     0&     1\\
};

\sfrm{table}{2}{3};

\node[above=2pt of table-1-2]{$\delta$};
\node[base right=2pt of table-1-3]{{$\zeta_{j}$}};
\node[base right=2pt of table-2-3]{$\beta_{j}$};

\end{tikzpicture}
\end{center}
\end{example}

\begin{example} \label{ex3} Consider the string datum given by the word
$$
w = \xi_0 - \gamma \sim \delta - \zeta_{j_1} \sim \beta_{j_1} - \delta \sim \gamma - \xi_i \sim \alpha_i - \gamma \sim \delta
- \beta_{j_2} \sim \zeta_{j_2} - \delta \sim \gamma
$$
Then the corresponding canonical forms $\bigl(\Theta_x(w), \Theta_y\bigr(w))$ are the following:
\begin{center}
\begin{tikzpicture}

\matrix[tbl5,text width=15pt, minimum height=15pt,  name=table] at(0,0)
{
      1&     0&     0&      0\\
      0&     1&     0&      0\\
      0&     0&     1&      0\\
};

\sfrm{table}{3}{4};
\node[above=2pt of table-1-2.north east]{$\gamma$};
\node[base left=2pt of table-1-1]{{$\xi_{0}$}};
\node[base left=2pt of table-2-1]{{$\xi_{i}$}};
\node[base left=2pt of table-3-1]{$\alpha_{i}$};

\matrix[tbl5,text width=15pt, minimum height=15pt,  name=table] at(3,0)
{
      1&     0&     0&      0\\
      0&     0&     0&      1\\
      0&     1&     0&      0\\
      0&     0&     1&      0\\
};

\sfrm{table}{4}{4};

\node[above=2pt of table-1-2.north east]{$\delta$};
\node[base right=2pt of table-1-4]{{$\zeta_{j_1}$}};
\node[base right=2pt of table-2-4]{{$\zeta_{j_2}$}};
\node[base right=2pt of table-3-4]{$\beta_{j_1}$};
\node[base right=2pt of table-4-4]{$\beta_{j_2}$};
\end{tikzpicture}
\end{center}
\end{example}

\begin{example} \label{ex4} Consider the string datum given by the word
$$
w = \alpha_\infty - \gamma \sim \delta - \beta_j \sim \zeta_j - \delta \sim \gamma - \xi_i \sim \alpha_i - \gamma \sim \delta - \beta_\infty.
$$
Then the corresponding canonical forms $\bigl((\Theta_x(w), \Theta_y(w)\bigr)$ are the following:

\begin{center}
\begin{tikzpicture}

\matrix[tbl5,text width=15pt, minimum height=15pt,  name=table] at(0,0)
{
      0&     1&     0\\
      1&     0&     0\\
      0&     0&     1\\
};

\sfrm{table}{3}{3};
\node[above=2pt of table-1-2]{$\gamma$};
\node[base left=2pt of table-1-1]{$\xi_{i}$};
\node[base left=2pt of table-2-1]{$\alpha_{\infty}$};
\node[base left=2pt of table-3-1]{$\alpha_{i}$};

\matrix[tbl5,text width=15pt, minimum height=15pt,  name=table] at(3,0)
{
      0&     1&     0\\
      0&     0&     1\\
      1&     0&     0\\
};

\sfrm{table}{3}{3};

\node[above=2pt of table-1-2]{$\delta$};
\node[base right=2pt of table-1-3]{{$\zeta_{j}$}};
\node[base right=2pt of table-2-3]{$\beta_{\infty}$};
\node[base right=2pt of table-3-3]{$\beta_{j}$};
\end{tikzpicture}
\end{center}
\end{example}

\subsection{\texorpdfstring{Indecomposable Cohen--Macaulay modules over $\RP_{\infty \infty}$}{Indecomposable CM modules over the maximal degeneration of type P}}
Consider an object  $\Tau = \bigl(N, V, (\Theta_x, \Theta_y)\bigr)$ of the category $\Tri(\RA)$. Then we may assume that
\begin{itemize}
\item $N = X_{i_1} \oplus \dots \oplus X_{i_k} \oplus Y_{j_1} \oplus \dots \oplus Y_{j_l} \subseteq \RR^{k+l}$ for
certain indices $i_1, \dots, i_k, j_1, \dots, j_l \in \mathbb{N}_{0} \cup \{\infty\}$, see the beginning of Section \ref{S:explizit} for the definition of $\RR$--modules $X_i$ and $Y_i$ for $i \in \mathbb{N}_0 \cup \{\infty\}$.
\item $V = \kk^t$ for some $t \in \NN_0$.
\end{itemize}
According to Theorem \ref{T:BDdimOne}, the corresponding Cohen--Macaulay $\RA$--module $M = \GG(\Tau)$ is determined
 by the following commutative diagram in $\mod(\RA)$:
 \begin{equation}
\begin{array}{c}
\xymatrix@M+1pt
{ 0 \ar[r] & \idm  \mM  \ar[r] \ar@{=}[d] & M \ar[r]^-{\sigma} \ar[d]   & \kk^t
\ar[d]^-{\tilde{\theta}} \ar[r] & 0 \\
0 \ar[r] & \idm  \mM  \ar[r] & \mM
\ar[r]^-{\pi} & N/\idm N
\ar[r] & 0.
}
\end{array}
\end{equation}

\begin{lemma}\label{generators1}
Let $\bigl\{e_1, \dots, e_t\bigr\}$ be the standard basis of $\kk^t$. For any $1 \le i \le t$ choose $w_i \in N$ such that
$\pi(w_i) = \tilde{\theta}(e_i)$. Then we have:
\begin{equation}\label{E:triplegivesmodule}
M = \bigl\langle w_1, \dots, w_t\bigr\rangle_\RA  \subseteq \RR^{k+l}
\end{equation}
and $t$ is the minimal number of generators of $M$.
\end{lemma}

\begin{proof}
By definition of $M$, for any $1 \le i \le t$ we have: $w_i \in M$. Next, the induced map
$\bar\sigma\!: M/\idm M \lar \kk^t$ is an isomorphism (see the proof of \cite[Theorem 12.5]{BDNonIsol}) and
$\bar\sigma(w_i) = e_i$. Hence,  $\bigl\{\bar{w}_1, \dots, \bar{w}_t\bigr\}$ is a basis of $M/\idm M$ and
(\ref{E:triplegivesmodule})  follows from Nakayama's Lemma.
\end{proof}

\begin{lemma}\label{L:comparison}
Let $\Tau = \bigl(N, \kk^t, (\Theta_x, \Theta_y)\bigr)$ be an indecomposable object of $\Tri(\RA)$ as above and
$M = \GG(\Tau) = \bigl\langle w_1, \dots, w_t\bigr\rangle_\RA \subseteq \RR^{k+l}$. Then the following results are true.
\begin{enumerate}
\item We have either
\begin{itemize}
\item  $\Tau \cong \Bigl(\RR, \kk, \bigl((1), (1)\bigr)\Bigr)$ (in this case $M\cong \RA$), or
\item $\II(M) = \bigl\langle w_1, \dots, w_t\bigr\rangle_{\RP} \subseteq \RR^{k+l}$.
\end{itemize}
In both cases we have:
$$
\JJ(M) = \bigl\langle w_1, \dots, w_t, vw_1, \dots, vw_t, uw_1, \dots, uw_t\bigr\rangle_{\RT} \subseteq \RR^{k+l}.
$$
\item The Cohen--Macaulay module $M$, respectively $\II(M)$ and $\JJ(M)$, is  locally free on the punctured spectrum of $\RA$, respectively $\RP$ and $\RT$, if and only if $N$ contains no direct summands isomorphic to $X_\infty$ or $Y_\infty$.
\end{enumerate}
\end{lemma}

\begin{proof} (1) The statement about $\II(M)$ is a corollary of Proposition \ref{P:genrestrict}. The statement about $\JJ(M)$ follows from the fact that $\JJ(\RT) = \langle 1, u, v \rangle_\RA \subset Q$ in $\RA-\mod$.

\smallskip
\noindent
(2) The Cohen--Macaulay module $M$ is locally free on the punctured spectrum of $\RA$ if and only $M \cong \RR \boxtimes_\RA M$
is locally free on the punctured spectrum of $\RR$, see Remark \ref{R:locfree}. The latter is equivalent  that $N$ contains no direct summands isomorphic to  $X_\infty$ and $Y_\infty$.  Since the rational envelopes of $M$, $\II(M)$ and $\JJ(M)$ are the same,
the result follows.
\end{proof}

\noindent
Now we are ready to state the final classification of the indecomposable Cohen--Macaulay $\RP$--modules.
For any $l \in \NN$ introduce the letters $\bx^{\pm}_l$ and $\by^{\pm}_l$
as well as $\bx_{\scriptscriptstyle  0}, \bx_\infty, \by_{\scriptscriptstyle  0}$ and $\by_{\scriptscriptstyle \infty}$.

\begin{definition}\label{band2}
 \emph{Band modules} $B = B(\omega, m, \lambda)$ have the following combinatorics:
\begin{itemize}
\item $\omega = \bx_{i_1}^{\sigma_1} \by_{j_1}^{\tau_1} \dots \bx_{i_n}^{\sigma_n} \by_{j_n}^{\tau_n}$ is a \emph{non--periodic} word, where $\sigma_k, \tau_k \in \{+, -\}$ and $i_k, j_k \in \NN$ for  $1 \le k \le n$.
\item $m \in \NN$ and $\lambda \in \kk^\ast$.
\end{itemize}
Consider the following Cohen--Macaulay $\RR$--module
$$
N := X_{i_1}^{\oplus m} \oplus Y_{j_1}^{\oplus m} \oplus \dots \oplus X_{i_n}^{\oplus m} \oplus Y_{j_n}^{\oplus m} \subseteq
\RR^{2mn}.
$$
Then $B$ is the following $\RA$--submodule of $N$:
\begin{equation}\label{E:band}
B:=
\left[
\begin{array}{ccccc}
\left(
\begin{array}{c}
f''_{i_1} I \\
g'_{j_1} I \\
0 \\
0 \\
\vdots \\
0 \\
0
\end{array}
\right),   &
\left(
\begin{array}{c}
0 \\
g''_{j_1} I \\
f'_{i_2} I \\
0 \\
\vdots \\
0 \\
0
\end{array}
\right), &  \dots, \, &
\left(
\begin{array}{c}
0 \\
0 \\
0 \\
0 \\
\vdots \\
f''_{i_n} I \\
g'_{j_n} I
\end{array}
\right), &
\left(
\begin{array}{c}
f'_{i_1} J \\
0 \\
0 \\
0 \\
\vdots \\
0\\
g''_{j_n} I
\end{array}
\right)
\end{array}
\right]_{\RA}
\subseteq N,
\end{equation}
where for any $1 \le k \le n$ the elements $f'_{i_k}, f''_{i_k}, g'_{j_k}$ $g''_{j_k}$ are defined by the tables:
\begin{equation}\label{E:generators}
\begin{array}{|c|c|c|}
\hline
\sigma_k & f'_{i_k} & f''_{i_k} \\
\hline
+ & u & x^{i_k}  \\
- & x^{i_k} & u  \\
\hline
\end{array} \qquad
\begin{array}{|c|c|c|}
\hline
\tau_k & g'_{j_k} & g''_{j_k} \\
\hline
 + & v & y^{j_k} \\
 - & y^{j_k} & v \\
\hline
\end{array}
\end{equation}
\end{definition}

\begin{definition}\label{string2}
A \emph{string module} $S = S(\omega)$ is defined by a word $\omega$ of the following form.
$\omega$ has a beginning, an intermediate part and an end. The beginning as well as the end may consist of zero, one or two letters.
The following table lists all possible beginnings and ends for $\omega$ (any beginning from the first column can match
any ending from the last column):
\begin{align*}
\begin{tabular}{| r |  c  | l |}
\hline
$\hat{\bz}_{i_{0}}  \check{\bz}_{j_0}^{\tau_0}$ &
$\text{intermediate part}$ &
$\hat{\bz}_{i_n}^{\sigma_n}  \check{\bz}_{j_n}$ \\
\hline \hline
$\text{void}$& \multirow{6}{*}{$\bx_{i_1}^{\sigma_1} \by_{j_1}^{\tau_1} \dots \bx_{i_{n-1}}^{\sigma_{n-1}} \by_{j_{n-1}}^{\tau_{n-1}}$} &$ \text{void}$\\
$	 \by_{\scriptscriptstyle  0} \, $& &$ \bx_{\scriptscriptstyle  0} $\\
$	 \by_{\scriptscriptstyle \infty} $& &$ \bx_{\scriptscriptstyle  \infty}  $\\
$	 \by_{j_0}^{\tau_0} $& &$ \bx_{i_n}^{\sigma_n}  $\\
$\bx_{\scriptscriptstyle  0}  \by_{j_0}^{\tau_0}$ & & $\bx_{i_n}^{\sigma_n}  \by_{\scriptscriptstyle 0} $\\
$\bx_{\scriptscriptstyle \infty}  \by_{j_0}^{\tau_0}$ & & $\bx_{i_n}^{\sigma_n}  \by_{\scriptscriptstyle \infty} $\\ \hline
\end{tabular}
\end{align*}
where
\begin{itemize}
\item
 $n \in \mathbb{N}$.
For $n = 1$ the intermediate part of $\omega$ is void.
\item
For any $0 \le k \le n$ we have: $i_k, j_k \in \mathbb{N}$ and $\sigma_k, \tau_k \in \bigl\{+, -\bigr\}$. 
\end{itemize}
In other words, a string word $\omega$ is given by an alternating sequence of letters $\bx_i$ or $\by_j$ such that a letter of the form $\bx_{\scriptscriptstyle  0}, \bx_{\scriptscriptstyle \infty}, \by_{\scriptscriptstyle  0}$ or $\by_{\scriptscriptstyle \infty}$ may only occur as the first or last letter of $\omega$.

\noindent Consider the Cohen--Macaulay $\RR$--module
\begin{align*}
N &= \hat{Z}_{i_0} \oplus \check{Z}_{j_0} \oplus X_{i_1} \oplus Y_{j_1} \oplus \dots  \oplus X_{i_{n-1}} \oplus Y_{j_{n-1}} \oplus \hat{Z}_{i_n} \oplus \check{Z}_{j_n},
\end{align*}
where for each $i\in \bigl\{i_0, i_n\bigr\}$ and each $j\in \bigl\{j_0, j_n\bigr\}$, we set
\begin{equation*}
\hat{Z}_{i} = \begin{cases} X_i & \text{if }\bz_{i} = \bx_{i},\\
0 & \text{if } \bz_{i} \text{ is void}.
\end{cases}
\qquad \qquad
\check{Z}_{j} = \begin{cases} Y_j & \text{if }\bz_{j} = \by_{j},\\
0 & \text{if } \bz_{j} \text{ is void}.
\end{cases}
\end{equation*}
\medskip
\noindent
Then $S=S(\omega)$ is the following  $\RA$--submodule of $N$:
\begin{align}\label{E:string}
S:=
\left[\begin{array}{ccccc}
\left(
\begin{array}{c}
f''_{i_0}  \\
g'_{j_0}  \\
0 \\
\vdots \\
0 \\
0 \\
0
\end{array}
\right),   &
\left(
\begin{array}{c}
0 \\
g''_{j_0}  \\
f'_{i_1}  \\
\vdots \\
0 \\
0 \\
0
\end{array}
\right), &  \dots, \,  &
\left(
\begin{array}{c}
0  \\
0  \\
0 \\
\vdots \\
g''_{j_{n-1}} \\
f'_{i_n} \\
0
\end{array}
\right),
&
\left(
\begin{array}{c}
0  \\
0  \\
0 \\
\vdots \\
0 \\
f''_{i_n} \\
g'_{j_n}
\end{array}
\right)
\end{array}
\right]_\RA \subseteq N
\end{align}
where for any $1 \le k \le n$, the elements $f'_{i_k}, f''_{i_k}, g'_{j_k}$ and $g''_{j_k}$ are defined by
the tables:
\begin{align}\label{strings-gen1}
\begin{array}{|c|c|c|}
\hline
\sigma_k & f'_{i_k} & f''_{i_k} \\
\hline
+ & ux & x^{i_k+1}  \\
- & x^{i_k+1} & ux  \\
\hline
\end{array}\quad
\begin{array}{|c|c|c|}
\hline
\tau_k & g'_{j_k} & g''_{j_k} \\
\hline
 + & vy & y^{j_k+1} \\
 - & y^{j_k+1} & vy \\
\hline
\end{array}
\end{align}
{The remaining entries are defined as follows:}
\begin{align}
\begin{array}{rl}
\begin{array}{|l|c|c|} \hline
& f''_{i_0}  \\ \hline
i_0 = 0  & x  \\
i_0 = \infty & ux \\
\hat{\bz}_{i_0} \text{ is void} & 0 \\ \hline
\multicolumn{2}{c}{}
\end{array} & \quad
\begin{array}{|l|c|c|} \hline
& g'_{j_n} \\ \hline
j_n = 0 & y \\
j_n = \infty & vy \\
\check{\bz}_{j_n} \text{ is void} & 0 \\ \hline
\multicolumn{2}{c}{}
\end{array} \\
\begin{array}{|l|l|c|c|}
\hline
& \sigma_n & f'_{i_n} & f''_{i_n}\\
\hline
i_n = 0 & & x &  0 \\
i_n \in \NN & + & ux & x^{j_n +1} \\
i_n \in \NN & - & x^{j_n+1} & ux \\
i_n = \infty & & ux & 0 \\
\hline
\end{array}
& \quad
\begin{array}{|l|l|c|c|}
\hline
& \tau_0 & g'_{j_0} & g''_{j_0}\\
\hline
j_0 = 0 & & 0 &  y \\
j_0 \in \NN & + & vy & y^{j_0+1} \\
j_0 \in \NN & - & y^{j_0+1} & vy \\
j_0 = \infty & & 0 & vy \\
\hline
\end{array}
\end{array}
\end{align}
\end{definition}

\begin{remark}
If the string parameter $\omega$ contains neither $\bx_{\scriptscriptstyle  0}$ nor $\by_{\scriptscriptstyle  0}$, 
there is a
better presentation of the module $S(\omega)$: we divide all entries of type $f'_i$ or $f''_i$ of $S(\omega)$ by $x$ and all entries of type $g'_j$ or $g''_j$ of $S(\omega)$ by $y$.
\end{remark}

\begin{remark} \label{merge}
Any string module $S$ in \eqref{E:string}
has a more compact presentation by ``merging'' every odd row with its subsequent row:
\begin{align*}
\begin{footnotesize}
S \cong
\left[\begin{array}{cccccc}
\left(
\begin{array}{c}
f''_{i_0}  + g'_{j_0}  \\
0 \\
\vdots \\
0 \\
0
\end{array}
\right),   &
\left(
\begin{array}{c}
g''_{j_0}  \\
f'_{i_1}  \\
\vdots \\
0 \\
0
\end{array}
\right), &
\left(
\begin{array}{c}
0  \\
f''_{i_1}+ g'_{j_1} \\
\vdots \\
0 \\
0
\end{array}\right), &
 \dots, \,  &
\left(
\begin{array}{c}
0  \\
0 \\
\vdots \\
g''_{j_{n-1}} \\
f'_{i_n}
\end{array}
\right),
&
\left(
\begin{array}{c}
0  \\
0  \\
\vdots \\
0 \\
f''_{i_n} + g'_{j_n}
\end{array}
\right)
\end{array}
\right]_\RA
\end{footnotesize}
\end{align*}
The same can be done with the horizontal stripes of any band module $B$ in  \eqref{E:band}:
\begin{align*} \begin{footnotesize}
B \cong
\left[
\begin{array}{ccccc}
\left(
\begin{array}{c}
(f''_{i_1} + g'_{j_1}) I \\
0 \\
\vdots \\
0
\end{array}
\right),   &
\left(
\begin{array}{c}
g''_{j_1} I \\
f'_{i_2} I \\
\vdots \\
0
\end{array}
\right),
 \dots, \, &
\left(
\begin{array}{c}
0 \\
0 \\
\vdots \\
(f''_{i_n} +
g'_{j_n}) I
\end{array}
\right), &
\left(
\begin{array}{c}
f'_{i_1} J \\
0 \\
\vdots \\
g''_{j_n} I
\end{array}
\right)
\end{array}
\right]_{\RA}
\end{footnotesize}
\end{align*}
\end{remark}

\begin{theorem}\label{T:main2} For the ring $\RA = \kk\llbracket x, y, u, v\rrbracket/(xy, xv, yu, uv, u^2, v^2)$ the classification of the indecomposable objects of $\CM(\RA)$ is the following.
\begin{itemize}
\item The modules $B(\omega, m, \lambda)$ and $S(\omega)$ are indecomposable. Moreover, any indecomposable Cohen--Macaulay $\RA$--module is isomorphic to some band or some string module.
\item $B(\omega, m, \lambda) \not\cong S(\omega')$  for any choice of parameters $\omega, \omega', m$ and $\lambda$.
\item $S(\omega) \cong S(\omega')$ if and only if $\omega' = \omega$ or $\omega' = \omega^{o}$, where $\omega^{o}$ is the opposite word.
\item  $B(\omega, m, \lambda) \cong B(\omega', m', \lambda')$ if and only if $m = m'$, $\lambda = \lambda'$ and $\omega'$ is given by a (possibly trivial) cyclic shift on all letters of $\omega$.
\end{itemize}
\end{theorem}
\begin{proof}
According to  Theorem \ref{T:main1} the classification problem of indecomposable objects of $\CM(\RA)$
is equivalent to the matrix problem over the bunch of chains $\dX$ from Definition  \ref{D:distinguishedbuch}. More precisely, we had a diagram of categories and functors
\begin{equation*}
\begin{array}{c}
\xymatrix@M+1pt{
\CM(\RA) \ar@{}[r]|-{\sim} \ar@/^/[r]^-{\FF}  & \Tri(\RA) \ar[r]^{\PP} \ar@/^/[l]^-{\GG}
& \Rep(\dX)
}
\end{array}
\end{equation*}
where $\FF$ and $\GG$ are mutually inverse equivalences of categories and $\PP$ is a full functor preserving
isomorphy classes and indecomposability of objects.

\medskip
\noindent
1.~Indecomposable objects of the category $\Rep(\dX)$ are classified by string and band data according to  Theorem \ref{T:BondarenkoClassific}. Moreover, the indecomposable objects of $\Rep(\dX)$ lying in the essential image of $\PP$ are described by Remark \ref{R:EssentImage}.

\medskip
\noindent
2.~Let $w$ be the word of string or band datum in $\dX$, see Definition \ref{string1} and Definition \ref{band1}.
Note that $w$ is uniquely determined by the symbols in $\dE$ it contains.
It follows that we may delete all subwords of the form $\gamma \sim \delta$ or $\delta \sim \gamma$ and relations $-$ \emph{without loss of information}. Now we can translate the remaining subwords as follows:
\begin{align*}
\begin{array}{|c|c|c|c|c|c|c|c|}
\hline
\xi_0 	&  \alpha_i \sim \xi_i & \xi_i \sim \alpha_i  & \alpha_\infty &
\zeta_0 & \beta_j \sim \zeta_j & \zeta_j \sim \beta_j & \beta_\infty \\
\hline
\bx_{\scriptscriptstyle  0}	&	 	\bx_{i}^{+}	& \bx_{i}^{-}		& \bx_{\scriptscriptstyle \infty} &	
\by_{\scriptscriptstyle  0} &\by_{j}^{+} & \by_{j}^{-} & \by_{\scriptscriptstyle \infty} \\
\hline
\end{array}
\end{align*}
This table allows to pass from a word $w$ to a word $\omega$ as  in Definitions \ref{band2} and \ref{string2}.

\medskip
\noindent
3.~Consider a string datum $(w)$ (obeying the constraint from Remark \ref{R:EssentImage})
or a band datum $(w, m, \lambda)$. In Subsection \ref{dXofP} we explained the construction of the corresponding indecomposable object $\Theta = \bigl(\Theta_x, \Theta_y)$ of $\Rep(\dX)$. Now we give  the construction
of a triple $\Tau=(N,V,\Theta)$ in $\Tri(\RA)$ such that $\PP(\Tau) = \Theta$.
 Let $m_0$, $m_i^{\pm}$, $m_\infty$, $n_0$, $n_j^{\pm}$ respectively $n_\infty$ be the number of times the letter
$\bx_{\scriptscriptstyle  0}$, $\bx_{i}^{\pm}$, $\bx_{\scriptscriptstyle \infty}$, $\by_{\scriptscriptstyle  0}$, $\by_{j}^{\pm}$ respectively $\by_{\scriptscriptstyle \infty}$ occurs in $w$.
Let $t$ be the number of times $\gamma$ (or $\delta$) occurs in $w$.
Then $\Tau = (N,V, \Theta)$, where
\begin{itemize}
\item $N = X_0^{m_0} \oplus \bigoplus_{i \in \NN} X_i^{m_i^+ + m_i^-} \oplus X_\infty^{m_\infty}
\oplus Y_0^{n_0} \oplus \bigoplus_{j \in \NN} Y_j^{n_j^+ + n_j^-} \oplus Y_\infty^{n_\infty}$,
\item $V= \kk^t$,
\item $\Theta= \bigl(\Theta_x, \Theta_y)$.
\end{itemize}

\medskip
\noindent
4.~ Now recall the construction of the indecomposable Cohen--Macaulay $\RA$--module $M = \GG(\Tau)$.
Consider a basis of a $\kk$--vector space $N/\idm N$ given by the images of the distinguished generators of the indecomposable direct summands of $N$.
Let $\pi\!: N \to N/\idm N$ be the canonical projection and $\widetilde\Theta:= \Bigl(\begin{smallmatrix}\displaystyle \Theta_x \\ \displaystyle \Theta_y\end{smallmatrix}\Bigr)\!: V \to N/\idm N$.
 By Theorem \ref{T:BDdimOne} we have: $$
 M := \GG(\Tau) = \pi^{-1}\bigl(\im (\widetilde\Theta) \bigr) \subseteq N.$$
To compute $M$ we do the following:
\begin{enumerate}
\item We multiply each entry of $\widetilde\Theta$ with its horizontal weight.
\item For each subword $\epsilon' \sim \epsilon''$ in $w$ such that $\epsilon', \epsilon'' \in \dE$ we merge the two corresponding rows in $\widetilde\Theta$ and add their entries to each other.
\item We translate all entries of the new matrix $\widetilde\Theta$ as follows:
$$
\begin{array}{|c|c|c|c|c|c|c|c|}
\hline
 \xi_0 & \xi_k & \alpha_\infty & \alpha_k & \zeta_0 & \zeta_k & \beta_\infty & \beta_k \\ \hline
 e_1 & x^k & u &  u & e_2 & y^k & v &  v \\ \hline
\end{array} \qquad \displaystyle{k \in \NN}.
$$
\end{enumerate}
Lemma \ref{L:comparison} implies that
$M$ is generated by the columns of the modified matrix $\widetilde\Theta$.
Finally, observe that $M \cong (x+y) M$ in $\RA-\mod$. Thus, if $w$ is a string word containing one of the symbols
 $\xi_0$ or $\zeta_0$, one has to multiply the  columns of $\widetilde\Theta$ with $(x+y)$ to obtain entries which lie in $\RA$.
After permutation of rows in $\tilde\Theta$, one obtains exactly the presentations $(\ref{E:band})$ respectively $(\ref{E:string})$ for the words in Definition $\ref{band2}$ or $\ref{string2}$.

\medskip
\noindent
5. The statement about the isomorphy classes of string modules in $\CM(\RA)$
is a direct translation of the corresponding result for the category $\Rep(\dX)$ stated in  Theorem \ref{T:BondarenkoClassific}.
Considering all pairwise non-equivalent band data $(w,m,\lambda)$, we may assume that the last letter of $w$ is $\gamma$ or $\delta$ by the equivalence conditions in Definition \ref{band1}. Then Theorem $\ref{T:BondarenkoClassific}$ yields the isomorphism conditions for band modules $B(\omega,m,\lambda)$ as stated above.

\medskip
\noindent
6.~Summing up, the key point of the proof of Theorem \ref{T:main2} is that we have the following isomorphisms in the category $\Rep(\dX)$ \emph{by construction}:
$$
\PP \circ \Bigl(\FF\bigl(B(\omega, m, \lambda)\bigr)\Bigr) \cong \bigl(\Theta_x(w, m, \lambda), \Theta_x(w, m, \lambda)\bigr)
$$
for a band module $B(\omega, m, \lambda)$ from Definition \ref{band2} and
$$
\PP \circ \Bigl(\FF\bigl(S(\omega)\bigr)\Bigr) \cong \bigl(\Theta_x(w), \Theta_x(w)\bigr).
$$
for a string module $S(\omega)$ from Definition \ref{string2}.
\end{proof}

\begin{remark}\label{R:locfreePSpec} By Lemma \ref{L:comparison}
any band module $B(\omega, m, \lambda)$ is locally free on the punctured spectrum. A string module $S(\omega)$ is \emph{not}
locally free on the punctured spectrum if and only if
$\omega$ contains a letter $\bx_{\scriptscriptstyle \infty}$ or $\by_{\scriptscriptstyle \infty}$.
\end{remark}

\begin{example} \label{ex:A} In the following examples we translate the canonical forms of the preceding subsection into indecomposable Cohen--Macaulay modules over $\RA$ as described in the proof of Theorem \ref{T:main2}.
\begin{enumerate}
\item Let $(w,m,\lambda)$ be the band datum from Example \ref{ex1}.\\
Then $w$ corresponds to $\omega = \bx^-_{i} \by^-_{j_1} \bx^+_{i} \by^-_{j_2}$
and its band module is given by
\begin{align*}
\qquad B(\omega,m,\lambda) &\cong
\left[\begin{array}{cccc}
\left(
\begin{array}{c}
x^i I  \\
0  \\
0 \\
v J \\
\end{array}
\right),
&
\left(
\begin{array}{c}
u I \\
0  \\
y^{j_1} I \\
0 \\
\end{array}
\right),
&
\left(
\begin{array}{c}
0  \\
u I \\
v I \\
0 \\
\end{array}
\right),
&
\left(
\begin{array}{c}
0  \\
x^i I \\
0 \\
y^{j_2} I \\
\end{array}
\right)
\end{array}\right]_\RA \\
& \cong
\left[\begin{array}{cccc}
\left(
\begin{array}{c}
x^i I   \\
v J
\end{array}
\right),
&
\left(
\begin{array}{c}
(u  +
y^{j_1}) I \\
0 \\
\end{array}
\right),
&
\left(
\begin{array}{c}
v I \\
u I
\end{array}
\right),
&
\left(
\begin{array}{c}
0  \\
(x^i  +
y^{j_2}) I
\end{array}
\right)
\end{array}\right]_\RA
\end{align*}
Here, $J$ denotes the Jordan block with eigenvalue $\lambda$ and $I$ the identity matrix, both of size $m$.
\item
Let $w$ be the string datum from Example \ref{ex2}.\\
Then $w$ corresponds to $\omega = \bx^-_{i} \by^-_{j}$
and its string module is given by
\begin{align*}
S(\omega) &\cong
\left[\begin{array}{ccc}
\left(
\begin{array}{c}
x^i   \\
0
\end{array}
\right),
&
\left(
\begin{array}{c}
u \\
y^{j}
\end{array}
\right),
&
\left(
\begin{array}{c}
0  \\
v
\end{array}
\right) \end{array} \right]_\RA \cong (x^i, u +y^j, v)_\RA
\end{align*}
\item
Let $w$ be the string datum from Example \ref{ex3}.\\
Then $w$ corresponds to $\omega = \bx_{\scriptscriptstyle  0} \by^-_{j_1} \bx^-_{i} \by^+_{j_2}$
and its string module is given by
\begin{align*}
S(\omega) &\cong
\left[\begin{array}{cccc}
\left(
\begin{array}{c}
x  \\
0  \\
y^{j_1+1} \\
0 \\
\end{array}
\right),
&
\left(
\begin{array}{c}
0 \\
x^{i+1}  \\
v y \\
0 \\
\end{array}
\right),
&
\left(
\begin{array}{c}
0  \\
u x \\
0 \\
v y \\
\end{array}
\right),
&
\left(
\begin{array}{c}
0  \\
0  \\
0 \\
y^{j_2+1} \\
\end{array}
\right)
\end{array}\right]_\RA \\
& \cong
\left[\begin{array}{cccc}
\left(
\begin{array}{c}
x  + y^{j_1+1} \\
0 \\
\end{array}
\right),
&
\left(
\begin{array}{c}
v y \\
x^{i+1}
\end{array}
\right),
&
\left(
\begin{array}{c}
0  \\
u x +
v y
\end{array}
\right),
&
\left(
\begin{array}{c}
0  \\
y^{j_2+1} \\
\end{array}
\right)
\end{array}\right]_\RA
\end{align*}
\item
Let $w$ be the string datum from Example \ref{ex4}. \\
Then $w$ corresponds to $\omega = \bx_{\scriptscriptstyle  \infty} \by^+_{j} \bx^-_{i} \by_{\scriptscriptstyle \infty}$
and its string module is given by
\begin{align*}
S(\omega) &\cong
\left[\begin{array}{ccc}
\left(
\begin{array}{c}
u  \\
0  \\
v \\
0
\end{array}
\right),
&
\left(
\begin{array}{c}
0 \\
x^{i} \\
y^j \\
0
\end{array}
\right),
&
\left(
\begin{array}{c}
0  \\
u  \\
0 \\
v
\end{array}
\right)
\end{array}\right]_\RA
 \cong
\left[\begin{array}{ccc}
\left(
\begin{array}{c}
u + v \\
0
\end{array}
\right),
&
\left(
\begin{array}{c}
y^j \\
x^i
\end{array}
\right),
&
\left(
\begin{array}{c}
0  \\
u  +
v
\end{array}
\right)
\end{array}\right]_\RA
\end{align*}
\end{enumerate}
\end{example}
\noindent
Remind that our original motivation was to describe indecomposable Cohen--Macaulay modules over
the ring $\RP = \kk\llbracket x,y, z\rrbracket/(xy, z^2)$. Theorem \ref{T:RejectionLemma}, Lemma \ref{L:comparison} and Theorem \ref{T:main2} yield a complete classification of indecomposable objects of
$\CM(\RP)$.

\begin{theorem}\label{C:main} An indecomposable Cohen--Macaulay $\RP$--module is either $\RP$, or one of the
 band modules  (\ref{E:band}) respectively  string modules  (\ref{E:string}).
Moreover, in the formulae  (\ref{E:band}) and (\ref{E:string}),
the generation over $\RA$ can be replaced by the generation over $\RP$ (with the only exception of $S(\bx_{\scriptscriptstyle  0} \by_{\scriptscriptstyle  0})$).
\end{theorem}
\begin{remark}\label{translationP}
Any string or band module $M$ over $\RA$ can be translated into a Cohen--Macaulay module $\II(M)$ over $\RP$ as follows.
\begin{enumerate}
\item Assume that $M= B(\omega,m,\lambda)$ or $M= S(\omega)$ where the string datum $\omega$ does not contain $\bx_{\scriptscriptstyle  0}$ or $\by_{\scriptscriptstyle  0}$.
We translate the entries of $M$ as follows:
\begin{align*}
\begin{array}{|c|c|c|c|}
\hline
x^i & u &  y^{j} & v \\
\hline
x^{i+1} & xz & y^{j+1} & yz \\
\hline
\end{array}
\end{align*}
There is a better presentation of $M$ if the following assumptions are satisfied:
\begin{itemize}
\item $\omega$ has even length and the sequence of signs $\tau_k$, $\sigma_k$ in $\omega$ alternates,
\item if $\omega$ begins (respectively ends) with $\bx_{\scriptscriptstyle \infty}$ or $\by_{\scriptscriptstyle \infty}$ then the first (respectively last) sign in $\omega$ is $+$ (respectively $-$).
\end{itemize}
In this case, we apply Remark \ref{merge} to $M$, replace $xz + yz$ by $z$, any entry of type $x^i$ by $x^{i-1}$ and any entry $y^j$ by $y^{j-1}$.
\item
If $M=S(\omega) \not \cong S(\bx_{\scriptscriptstyle  0} \by_{\scriptscriptstyle  0})$ is a
string module such that the string datum $\omega$ contains $\bx_{\scriptscriptstyle  0}$ or $\by_{\scriptscriptstyle  0}$, we translate all entries of $M$ as follows:
\begin{align*}
\begin{array}{|c|c|c|c|}
\hline
x^i & u x & y^{j} & v y \\
\hline
x^{i} & xz  & y^{j} & yz \\
\hline
\end{array}
\end{align*}
\end{enumerate}
\end{remark}

\begin{remark}
Theorem \ref{C:main} remains valid for any curve singularity of type $\RP_{2r+1,2s+1}$, where $r,s \in \mathbb{N}_0 \cup \{\infty\}$, but string and band modules have to be redefined in the following way:
\begin{enumerate}
\item The band and string modules over $\RP_{2r+1,\infty}$ are given by the Definitions \ref{band2} and \ref{string2}, but their string and band words $\omega$ may only contain letters $\bx_i$ such that $0 \leq i \leq r$
or $\by_j$, where $j \in \mathbb{N}_{0} \cup \{\infty \}$.
\item Band and string data over $\RP_{2r+1,2s+1}$, where $r,s \in \mathbb{N}_0$,
may only contain the letters $\bx_i$
such that $0 \leq i \leq r$ or $\by_j$ such that $0 \leq j \leq s$.
\end{enumerate}
The method of this section can also be generalized using Bondarenko's work \cite{bo1} to obtain an explicit classification of the indecomposable Cohen--Macaulay modules over the remaining curve singularities $\RP_{2r,q}$, where $r\in \mathbb{N}$ and $q\in \mathbb{N} \cup \{ \infty\}$.
This generalization is straightforward to carry out, but the explicit combinatorics are too space-consuming to be stated in the present article.
\end{remark}
\begin{example}
In the following we apply Remark \ref{translationP}
to translate the string and band modules over $\RA$ from \ref{ex:A} into indecomposable Cohen-Macaulay modules over $\RP$.
\begin{enumerate}
\item Let $(\omega,m,\lambda)$ be a band datum with $\omega = \bx^-_{i} \by^-_{j_1} \bx^+_{i} \by^-_{j_2}$.
Then its band module $B=B(\omega,m,\lambda)$ translates over $\RP$ into
\begin{align*}
\quad \II(B) &\cong
\left[\begin{array}{cccc}
\left(
\begin{array}{c}
x^{i+1}\, I   \\
yz \,J
\end{array}
\right),
&
\left(
\begin{array}{c}
(xz +
y^{j_1+1})\, I \\
0 \\
\end{array}
\right),
&
\left(
\begin{array}{c}
xz \,I \\
yz \,I
\end{array}
\right),
&
\left(
\begin{array}{c}
0  \\
(x^{i+1}  +
y^{j_2+1})\, I
\end{array}
\right)
\end{array}\right]_\RP
\end{align*}
\item
Let $\omega = \bx^-_{i} \by^-_{j}$.
Then the string module $S=S(\omega)$ translates over $\RP$ into
\begin{align*}
\II(S) &\cong
\left(
x^{i+1}, xz + y^{j+1},
yz \right)_\RP
\end{align*}
\item
Let $\omega = \bx_{\scriptscriptstyle  0} \by^-_{j_1} \bx^-_{i} \by^+_{j_2}$.
Then the string module $S=S(\omega)$ translates over $\RP$ into
\begin{align*}
\II(S)
& \cong
\left[\begin{array}{cccc}
\left(
\begin{array}{c}
x  + y^{j_1+1} \\
0 \\
\end{array}
\right),
&
\left(
\begin{array}{c}
yz \\
x^{i+1}
\end{array}
\right),
&
\left(
\begin{array}{c}
0  \\
(x + y)\, z
\end{array}
\right),
&
\left(
\begin{array}{c}
0  \\
y^{j_2+1} \\
\end{array}
\right)
\end{array}\right]_\RP
\end{align*}
\item
Let $\omega = \bx_{\scriptscriptstyle  \infty} \by^+_{j} \bx^-_{i} \by_{\scriptscriptstyle \infty}$.
Then the string module $S=S(\omega)$ translates over $\RP$ into
\begin{align*}
\quad \II(S) &\cong
\left[\begin{array}{ccc}
\left(
\begin{array}{c}
xz  \\
0  \\
yz \\
0
\end{array}
\right),
&
\left(
\begin{array}{c}
0 \\
x^i \\
y^j \\
0
\end{array}
\right),
&
\left(
\begin{array}{c}
0  \\
x z  \\
0 \\
y z
\end{array}
\right)\end{array}\right]_\RP
 \cong
\left[\begin{array}{ccc}
\left(
\begin{array}{c}
z \\
0
\end{array}
\right),
&
\left(
\begin{array}{c}
y^j \\
x^i
\end{array}
\right),
&
\left(
\begin{array}{c}
0  \\
z
\end{array}
\right)
\end{array}\right]_\RP
\end{align*}
\end{enumerate}
\end{example}

\subsection{\texorpdfstring{Induced Cohen-Macaulay modules of $\RT_{\infty \infty}$ and their matrix factorizations}{Induced CM modules over the maximal degeneration of type T and their matrix factorizations}}
\noindent Our next motivation was to study Cohen--Macaulay modules over the hypersurface singularity $T = \kk\llbracket a, b\rrbracket/(a^2 b^2)$.
At the beginning of Section \ref{S:explizit} we have constructed  a fully faithful  functor
$ \xymatrix@M+1pt{
\JJ\!: \CM(\RA) \ar@{^{(}->}[r]& \CM(\RT).
}
$
Its explicit description, adapted to the combinatorics of bands and strings, was explained in Lemma \ref{L:comparison}:
if $M = \langle w_1, \ldots, w_t \rangle_\RA \subseteq \RR^{k+l}$ then
$$\JJ(M) =
\langle w_1, \ldots, w_t, u w_1, \dots, u w_t, v w_1, \dots, v w_t \rangle_\RT \subseteq \RR^{k+l}.
$$
In fact, the number of generators of $\JJ(M)$ can be reduced.
\begin{remark}\label{translationT}
Any string or band module $M$ over $\RA$ can be translated into a Cohen--Macaulay module $\JJ(M)$ over $\RT$ as follows.
\begin{itemize}
\item Let $M$ be a band module or a string module $S(\omega)$ such that $\omega$ does not
$\bx_{\scriptscriptstyle  0}$ or $\by_{\scriptscriptstyle  0}$.
Then
$\JJ(M) = \langle w_1, \ldots, w_t \rangle_T$. Moreover,
we may replace all entries in every generator of $\JJ(M)$ by the table:
\begin{align*}
\begin{array}{|c|c|c|c|}
\hline
x^i & u &  y^{j} & v \\
\hline
a^{i+2} & a^2 b & b^{j+2} & a b^2 \\
\hline
\end{array}
\end{align*}
\item Let $M = S(\omega)$ such that $\omega$ contains $\bx_{\scriptscriptstyle  0}$ or $\by_{\scriptscriptstyle  0}$.
For every column  $w_i$ containing an entry $x$ respectively $y$, add $u w_i$ respectively $v w_i$ to the span $\langle w_1, \ldots, w_t \rangle_\RT$.
Then $\JJ(M)$ is equal to this span with the additional generators.
At last, we can translate all entries into elements of $\RT$ using the table:
\begin{align*}
\begin{array}{|c|c|c|c|}
\hline
x^i & u x & y^{j} & v y \\
\hline
a^{i+1} & a^2 b  & b^{j+1} & a b^2 \\
\hline
\end{array}
\end{align*}
\end{itemize}
\end{remark}
\begin{example}
Now we translate the string and band modules over $\RA$ from \ref{ex:A} into indecomposable Cohen-Macaulay modules over $\RT$ using Remark \ref{translationT}.
\begin{enumerate}
\item Let $(\omega,m,\lambda)$ be a band datum with $\omega = \bx^-_{i} \by^-_{j_1} \bx^+_{i} \by^-_{j_2}$.
Then its band module $B=B(\omega,m,\lambda)$ translates over $\RT$ into
\begin{align*}
\quad \JJ(B) &\cong
\left[\begin{array}{cccc}
\left(
\begin{array}{c}
a^{i+2}\, I   \\
ab^2 \,J
\end{array}
\right),
&
\left(
\begin{array}{c}
(a^2 b +
b^{j_1+2})\, I \\
0 \\
\end{array}
\right),
&
\left(
\begin{array}{c}
a b^2 \,I \\
a^2 b \,I
\end{array}
\right),
&
\left(
\begin{array}{c}
0  \\
(a^{i+2}  +
b^{j_2+2})\, I
\end{array}
\right)
\end{array}\right]_\RT
\end{align*}
\item
Let $\omega = \bx^-_{i} \by^-_{j}$.
Then the string module $S=S(\omega)$ translates over $\RT$ into
\begin{align*}
\JJ(S) &\cong
\left(
a^{i+2}, a^2 b + y^{j+2},
a b^2 \right)_\RT
\end{align*}
\item
Let $\omega = \bx_{\scriptscriptstyle  0} \by^-_{j_1} \bx^-_{i} \by^+_{j_2}$.
Then the string module $S=S(\omega)$ translates over $\RP$ into
\begin{align*}
\JJ(S)
& \cong
\left[\begin{array}{ccccc}
\left(
\begin{array}{c}
a^2  + b^{j_1+2} \\
0 \\
\end{array}
\right),
&
\left(
\begin{array}{c}
a b^2 \\
a^{i+2}
\end{array}
\right),
&
\left(
\begin{array}{c}
0  \\
a^2 b + a b^2
\end{array}
\right),
&
\left(
\begin{array}{c}
0  \\
b^{j_2+2} \\
\end{array}
\right),
&
\left(
\begin{array}{c}
a^2 b  \\
0 \\
\end{array}
\right)
\end{array}\right]_\RT
\end{align*}
\item
Let $\omega = \bx_{\scriptscriptstyle  \infty} \by^+_{j} \bx^-_{i} \by_{\scriptscriptstyle \infty}$.
Then the string module $S=S(\omega)$ translates over $\RP$ into
\begin{align*}
\quad \JJ(S) &\cong
\left[\begin{array}{ccc}
\left(
\begin{array}{c}
a^2 b + a b^2 \\
0
\end{array}
\right),
&
\left(
\begin{array}{c}
b^{j+2} \\
a^{i+2}
\end{array}
\right),
&
\left(
\begin{array}{c}
0  \\
a^2 b + a b^2
\end{array}
\right)
\end{array}\right]_\RT \\
& \cong
\left[\begin{array}{ccc}
\left(
\begin{array}{c}
ab \\
0
\end{array}
\right),
&
\left(
\begin{array}{c}
b^{j} \\
a^{i}
\end{array}
\right),
&
\left(
\begin{array}{c}
0  \\
ab
\end{array}
\right)
\end{array}\right]_\RT
\end{align*}
\end{enumerate}
\end{example}

\begin{remark}
There is an involution $\tau$ on $\RA =
\kk \llbracket x,y,u,v\rrbracket/(xy, xv, yu,  uv, u^2, v^2)$
which interchanges $x$ and $y$, $u$ and $v$.
Restricted to $\RP= \kk \llbracket x, y, z\rrbracket/(xy, z^2) \subset \RA$ , $\tau$ is still an involution such that $\tau(z) = z$.
The restriction of $\tau$ to $\RT = \kk \llbracket a,b \rrbracket/(a^2 b^2) \subset A$
interchanges $a$ and $b$. Overall, $\tau$ induces an involution on the category of Cohen--Macaulay modules over $\RA$, $\RP$ or $\RT$.
The corresponding action of $\tau$ on words $\omega$ of string or band data of $\CM(\RA)$
 is given by interchanging $\bx$ and $\by$ in $\omega$.
\end{remark}
\noindent
In the following table, we list all indecomposable Cohen--Macaulay ideals of $\RA$ and
$\RP$ (except $\RP$ itself) and the corresponding ideals of $\RT$ up to isomorphism and involution $\tau$.
Let $i, j \in \NN$ and $\lambda \in
\kk^\ast$. For all band data the multiplicity parameter $m$ is set to $1$.
\begin{center}
\begin{longtable}{|l|l|l|l|}
\hline
{invariant of $\dX$}											&	 {ideal in }$\RA$ 								& {ideal in }$\RP$										 & {ideal in }$\RT$ \\
\hline
\hline \endhead
$\bx_{\scriptscriptstyle  0}$													& $(x)$ 							 &	$(x)$													 & $(a^2)$ \\ \hline
$\bx_{\scriptscriptstyle \infty}$											& $(u)$								 &	$(xz)$												 & $(a^2 b)$ \\ \hline
$\bx_{i}^{-}$											& $(x^i, u)	$				 &$	(x^{i+1}, xz)	$							 & $(a^{i+2}, a^2 b)$ \\ \hline
$\bx_{\scriptscriptstyle  0} \by_{\scriptscriptstyle  0}$										& $(1)	$							 &$	(x+y, xz)	$									 &
$(a^2 + b^2, a^2 b, a b^2)$ \\ \hline
$\bx_{i}^{+}	\by_{\scriptscriptstyle  0}		$						& $(ux, x^{i+1}+y)$		&	 $(xz, x^{i+1} + y)$						 &
$(a^2 b, a^{i+2} + b^{2}, ab^2)$ \\ \hline
$\bx_{i}^{+}\by_{j}^{-}	$					& $(u, x^i + y^j, v)$	&	$(xz, x^{i+1} + y^{j+1}, yz)$	&
$(a^2 b, a^{i+2} + b^{j+2}, ab^2)$ \\ \hline
\multirow{2}{*}{$\bx_{i}^{-} \by_{\scriptscriptstyle  0}$}&
\multirow{2}{*}{$(x^{i+1}, ux + y)$}	&	
\multirow{2}{*}{$(x^{i+1}, xz + y)$}	&
$i=1: (a^{3}, a^2 b + b^{2})$ \\
&&& $i \geq 2: (a^{i+2}, a^2 b + b^{2}, ab^2)$ \\ \hline
$\bx_{i}^{-}\by_{j}^{-}	$					& $(x^i, u + y^j, v)$	&	 $(x^{i+1}, xz + y^{j+1}, yz)$	&
$(a^{i+2}, a^2 b + b^{j+2}, ab^2)$ \\ \hline
$\bx_{i}^{-}\by_{j}^{+}						$& $(x^i, u+v, y^j)$		&	 $(x^{i}, z, y^j)$							&
$(a^{i+2}, a^2 b+ ab^2, b^{j+2})$ \\ \hline
$\bx_{\scriptscriptstyle \infty}\by_{\scriptscriptstyle  0}$									& $(ux+y)					$	 &	$(xz+y)						$					 &
$(a^2 b + b^2, ab^2)$ \\ \hline
$\bx_{\scriptscriptstyle \infty}\by_{j}^{-}$						& $(u+y^j, v)$	&	 $(xz+y^{j+1}, yz)	$					 &
$(a^2 b + b^{j+2}, ab^2)$ \\ \hline
$\bx_{\scriptscriptstyle \infty}\by_{j}^{+}$						& $(u+v, y^j)$	&	 $(z, y^j)					$					 &
$(a^2 b + a b^2, b^{j+2})$ \\ \hline
$\bx_{\scriptscriptstyle \infty}\by_{\scriptscriptstyle \infty}$					& $(u+v)						$	 &	 $(z)							$						 &
$(a^2 b+ a b^2) \cong (ab)$\\ \hline
$(\bx_{i}^{-}\by_{j}^{-}, 1, \lambda)$						& $(x^i + \lambda v, y^j + u)$	& $(x^{i+1} + \lambda yz, y^{j+1} + xz)$		 &
$(a^{i+2} + \lambda a b^2, b^{j+2} + a^2 b)$ \\ \hline
$(\bx_{i}^{-}\by_{j}^{+}, 1, \lambda)			$			& $(x^i + \lambda y^j, u+v)$		& $(x^i + \lambda y^j, z)$								 &
$(a^{i+2} + \lambda  b^{j+2}, a^2 b + a b^2)$ \\ \hline
\end{longtable}
\end{center}
\begin{remark}
The above list does not contain all indecomposable ideals of $\RT$. For example,
the ideal $$(a^2 + \lambda b^2, a^2 b + a b^2)$$
is not induced by an ideal in $\RA$ for any $\lambda \in \kk^*$.
\end{remark}

\noindent Let $\underline{\MF}(a^2 b^2)$ be the homotopy category of  matrix factorizations of $a^2 b^2$.
By a result of Eisenbud \cite{Eisenbud} there is an equivalence of triangulated categories
$
\xymatrix@M+1pt{
\underline{\CM}(\RT) \ar[r]^-{\sim} & \underline{\MF}(a^2 b^2).}
$
In the following table, we list the matrix factorizations of $a^2 b^2$ which originate from an  indecomposable ideal in $\RA$ (up to isomorphism and involution). Let $i, j \in \NN$ and $\lambda \in \kk^\ast$.
\begin{center}
\begin{longtable}{|r|l|}
\hline
{ideal in $\kk \llbracket a,b \rrbracket /(a^2 b^2)$} & { matrix factorization $(\phi,\psi)$ of $a^2 b^2$} \\ \hline \hline
\endhead
$(a^2)$ &
$\begin{pmatrix}
b^2
\end{pmatrix}\,
\begin{pmatrix}
a^2
\end{pmatrix} $ \\ \hline
$(a^2 b)$ & $
\begin{pmatrix}
b
\end{pmatrix}\,
\begin{pmatrix}
a^2 b
\end{pmatrix}$\\ \hline
$(a^{i+2}, a^2 b)$&
$\begin{pmatrix}
b & 0 \\ -a^i & b
\end{pmatrix}\,
\begin{pmatrix}
a^2 b & 0 \\ a^{i+2} & a^2 b
\end{pmatrix}$ \\ \hline
$(a^{i+1}+  b^{j+1},  a^2 b, a b^2  )$ &$
\begin{pmatrix}
ab &0 &0 \\ -a^{i}  &b  &0  \\ -b^{j} & 0 & a
\end{pmatrix}\,
\begin{pmatrix}
ab &0 &0 \\ a^{i+1}  &a^2 b  &0  \\ b^{j+1} & 0 & a b^2
\end{pmatrix}$ \\ \hline
$(a^{i+3},   a^2 b +  b^{2}, a b^2 )$ &$
\begin{pmatrix}
b & 0 & 0 \\  - a^{i+1} & a b & 0 \\ 0 & -b & a
\end{pmatrix}\,
\begin{pmatrix}
a^2 b & 0 & 0 \\  a^{i+2} & ab & 0 \\ a^{i+1} b & b^2 & a b^2
\end{pmatrix} $\\ \hline
$(a^3, a^2b + b^2)$ &$
\begin{pmatrix}
ab & 0 \\ -a^2 & a^2b
\end{pmatrix}\,
\begin{pmatrix}
ab & 0 \\ -a &  b
\end{pmatrix}$ \\ \hline
$(a^{i+2},  a^2 b + b^{j+2},  a b^2 )$ &$
\begin{pmatrix}
b & 0 & 0 \\ -a^{i} & ab & 0 \\ -a^{i-1} b^j & -b^{j+1} & a
\end{pmatrix}\,
\begin{pmatrix}
a^2 b & 0 & 0 \\ a^{i+1} & ab & 0 \\ 0 & b^{j+2} & a b^2
\end{pmatrix} $\\ \hline
$(a^{i+2},  b^{j+2},   a^2 b  + a b^2 )$ &$
\begin{pmatrix}
b &0 &0 \\ 0  & a &0  \\ -a^{i} & -b^{j} & ab
\end{pmatrix}\,
\begin{pmatrix}
a^2 b &0 &0 \\ 0  &a b^2  &0  \\ a^{i+1} & b^{j+1} & ab
\end{pmatrix}$ \\ \hline
$( a^2 b + b^{j+1},  a b^2  )$ &$
\begin{pmatrix}
ab & 0 \\ -b^{j} & a
\end{pmatrix}\,
\begin{pmatrix}
ab & 0 \\ b^{j+1} & a b^2
\end{pmatrix}$ \\ \hline
$( b^{j+2}, a^2 b +  a b^2  )$ &$
\begin{pmatrix}
a & 0 \\ -b^{j} & ab
\end{pmatrix}\,
\begin{pmatrix}
a b^2 & 0  \\ b^{j+1} & ab
\end{pmatrix} $\\ \hline
$( ab )$ &$
\begin{pmatrix}
ab
\end{pmatrix}\,
\begin{pmatrix}
ab
\end{pmatrix} $\\ \hline
$(a^{i+2}+ \lambda  a b^2, a^2 b + b^{j+2} )$ &
$\begin{pmatrix}
ab & -b^{j+1} \\ -a^{i+1} & \lambda a b
\end{pmatrix}\,
\begin{pmatrix}
ab & \lambda^{-1} b^{j+1} \\ \lambda^{-1} a^{i+1}& \lambda^{-1} ab
\end{pmatrix}$ \\ \hline
$(a^{i+2}+ \lambda b^{j+2}, a^2 b + a b^2)$&$
\begin{pmatrix}
ab & 0 \\ -\lambda b^{j+1} - a^{i+1}& ab
\end{pmatrix}\,
\begin{pmatrix}
ab & 0 \\ \lambda b^{j+1} + a^{i+1}& ab
\end{pmatrix}$ \\ \hline
\end{longtable}
\end{center}

\begin{remark} By Kn\"{o}rrer's periodicity \cite{Knoerrer2} the functor
\begin{align*}
\begin{array}{rl}
{\underline{\MF}(a^2 b^2 ) \quad} \overset{\sim}{\longrightarrow} & \underline{\MF}(a^2 b^2 + uv)  \\
\begin{pmatrix}
\phi, \psi
\end{pmatrix} \quad
\longmapsto &
{\begin{pmatrix}
\phi & - u \cdot I \\
v \cdot I & \psi
\end{pmatrix}\,
\begin{pmatrix}
\psi & v \cdot I \\
-u\cdot I & \phi
\end{pmatrix}}
\end{array}
\end{align*}
 is an equivalence of triangulated categories.
It allows to get
 explicit families of matrix factorizations of any potential of type $$a^2 b^2 + u_1 v_1 + \ldots + u_d v_d \in \kk\llbracket a, b, u_1, \dots, u_d, v_1, \dots, v_d\rrbracket.$$
 \end{remark}

\begin{remark}
Let $\mathsf{char}( \kk) \neq 2$.
Then there is a ring isomorphism
$$\kk \llbracket a,b,c \rrbracket / ( a^2 b^2 - c^2 ) \cong \kk \llbracket x,y,z \rrbracket / (z^2 - xyz) =: \RT_{\infty \infty 2}.$$
The indecomposable Cohen--Macaulay modules over the surface singularity $\RT_{\infty \infty 2}$ have been classified in \cite{BDNonIsol}.
On the other hand, Kn\"{o}rrer's correspondence \cite{Knoerrer2} relates $\RT_{\infty \infty 2}$ to $\RT_{\infty\infty}$ by a restriction functor
\begin{align*}
{\underline{\MF}(a^2 b^2 - c^2)} {\longrightarrow} \underline{\MF}(a^2 b^2),
\end{align*}
such that every indecomposable matrix factorization of $a^2 b^2$ appears as a \emph{direct summand} of the restriction of some indecomposable matrix factorization of $a^2 b^2 - c^2$.
With some efforts, one can compute the matrix factorizations of $a^2 b^2$ corresponding to Cohen--Macaulay $\RT_{\infty \infty 2}$--modules of small rank.
However, it is not straightforward to derive all indecomposable matrix factorizations of $a^2 b^2$ by this approach.
\end{remark}

 \begin{remark}
 The approach to classify indecomposable  Cohen--Macaulay modules using the technique of tame matrix problems is close  in spirit to the study of torsion free sheaves on degenerations of elliptic curves. See \cite{Survey} for a survey of the corresponding results and methods.
 \end{remark}

\subsection{\texorpdfstring{Some remarks on  the stable category of Cohen--Macaulay modules}{Some remarks on  the stable category of CM modules}}

Let $(A, \idm)$ be a Gorenstein singularity (of any Krull dimension $d$).
By a result of Buchweitz \cite{Buchweitz}, the natural functor
$$
\underline{\CM}(A) \lar D_{sg}(A) := \frac{D^b\bigl(\mod(A)\bigr)}{\Perf(A)}
$$
 is an equivalence of triangulated categories.
If the singularity $A$ is not isolated, then $\underline{\CM}(A)$ is $\Hom$-infinite \cite{PhilNotes}. On the other hand,
the stable category of Cohen--Macaulay modules $\underline{\CM}^{\mathsf{lf}}(A)$  is always a  $\Hom$--finite
triangulated subcategory of $\underline{\CM}(A)$. By a result of Auslander \cite{PhilNotes}, the category $\underline{\CM}^{\mathsf{lf}}(A)$
is $(d-1)$--Calabi--Yau. This means that for any objects $M_1$ and $M_2$ of $\underline{\CM}^{\mathsf{lf}}(A)$ we have an isomorphism
$$
\underline{\Hom}_A(M_1, M_2) \cong \DD\bigl(\underline{\Hom}_A(M_2, \Sigma^{d-1}(M_1))\bigr),
$$
functorial in both arguments $M_1$ and $M_2$, where $\DD$ is the Matlis duality functor and $\Sigma = \Omega^{-1}$ is the suspension functor.
In particular, if $A$ is a Gorenstein curve singularity, then  for any $M \in \underline{\CM}^{\mathsf{lf}}(A)$ the algebra $\underline{\End}_A(M)$ is Frobenius. Thus, Theorem \ref{T:main1} gives a family of examples of representation tame
$0$--Calabi--Yau triangulated categories and
Theorem  \ref{C:main} provides a complete and  explicit description of indecomposable objects in one of such categories
 $\underline{\CM}^{\mathsf{lf}}(\RP)$ for $\RP = \kk \llbracket x, y, z\rrbracket/(xy, z^2)$.

\end{document}